\documentclass[english, 11pt]{amsart}
\usepackage{amssymb}
\usepackage{amsfonts}
\usepackage{amscd}
\usepackage[T1]{fontenc}
\usepackage{color}
\usepackage{tikz}
\usepackage{changes}
\usepackage{cancel}

\usepackage{mathrsfs} 

\def\deg{\operatorname{deg}}
\def\ac{{\overline{\rm ac}}}
\def\Supp{\operatorname{Supp}}

\def\Def{\operatorname{Def}}

\def\RDef{\operatorname{RDef}}

\def\LPas{\cL_{\rm DP}}

\def\11{{\mathbf 1}}
\def\AA{{\mathbb A}}

\def\CC{{\mathbb C}}

\def\FF{{\mathbb F}}

\def\LL{{\mathbb L}}

\def\NN{{\mathbb N}}

\def\QQ{{\mathbb Q}}

\def\ZZ{{\mathbb Z}}

\def\cC{{\mathcal C}}

\def\cF{{\mathcal F}}

\def\cL{{\mathcal L}}
\def\cM{{\mathcal M}}

\def\cO{{\mathcal O}}
\def\cP{{\mathcal P}}

\def\cU{{\mathcal U}}
\def\cV{{\mathcal V}}

\newcounter{dummy} \numberwithin{dummy}{section}

\newtheorem{thm}[dummy]{Theorem}
\newtheorem{lem}[dummy]{Lemma}
\newtheorem{cor}[dummy]{Corollary}
\newtheorem{prop}[dummy]{Proposition}
\newtheorem{claim}[dummy]{Claim}
\newtheorem{conj}[dummy]{Conjecture}

\newtheorem{maintheorem}[dummy]{Main Theorem}

\theoremstyle{definition}
\newtheorem{defn}[dummy]{Definition}

\newtheorem{def-prop}[dummy]{Proposition-Definition}
\newtheorem{def-theorem}[dummy]{Theorem-Definition}
\newtheorem{def-lem}[dummy]{Lemma-Definition}

\theoremstyle{remark}
\newtheorem{remark}[dummy]{Remark}

\theoremstyle{plain}

\numberwithin{equation}{section}

  {\par\medskip\noindent #1\par\begingroup%
    \advance\leftskip by 1em\advance\rightskip by 1em}%
  {\par\endgroup}

\DeclareMathOperator*{\Spec}{Spec}

\newcommand{\ord}{\operatorname{ord}}

\def\VF{\mathrm{VF}}

\def\VG{\mathrm{VG}}
\newcommand{\RF}{{\rm RF}}

\begin{document}

\author{Saskia Chambille}
\author{Kien Huu Nguyen}
\address{Saskia Chambille, KU Leuven, Department of Mathematics, Celestijnenlaan 200B, 3001 Heverlee, Belgium }
\email{saskia.chambille@kuleuven.be}
\address{Kien Huu Nguyen, Universit\'e Lille 1, Laboratoire Painlev\'e, CNRS - UMR 8524, Cit\'e Scientifique, 59655 Villeneuve d'Ascq Cedex, France}
\email{hkiensp@gmail.com}



\thanks{The authors are supported by the European Research Council under the European Community's Seventh Framework Programme (FP7/2007-2013) with ERC Grant Agreement nr. 615722
MOTMELSUM, and thank the Labex CEMPI  (ANR-11-LABX-0007-01).}

\title[]
{Proof of Cluckers-Veys's conjecture on exponential sums for polynomials with log-canonical threshold at most a half}

\begin{abstract} In this paper, we will give two proofs of the Cluckers-Veys conjecture on exponential sums for the case of polynomials in $\ZZ[x_{1},\ldots,x_{n}]$ having log-canonical thresholds at most one half. In particular, these results imply Igusa's conjecture and Denef-Sperber's conjecture under the same restriction on the log-canonical threshold.
\end{abstract}

\maketitle


\section{Introduction}\label{00}

Let $n\geq 1$ be a natural number and let $f\in\ZZ[x_{1},\ldots,x_{n}]$ be a non-constant polynomial in $n$ variables, for which we assume that $f(0,\ldots,0)=0$. For homogeneous polynomials $f$, Igusa has formulated, on page 2 of \cite{13}, a conjecture on the exponential sum
\[
E_{m,p}(f):=\dfrac{1}{p^{mn}}\sum_{\overline{x}\in(\ZZ/p^{m}\ZZ)^{n}}\exp\Big(\frac{2\pi if(x)}{p^{m}}\Big),
\]
where $p$ is a prime number and $m\in\NN$. More precisely, he predicted that there exist a constant $\sigma$, which depends on the geometric properties of $f$, and a positive constant $C$, independent of $p$ and $m$, such that for all primes $p$ and for all $m\geq 1$,
\[
|E_{m,p}(f)|\leq Cm^{n-1}p^{-m\sigma}.
\]
In particular, his conjecture implies an ad\`elic Poisson summation formula.\\

A local version of this sum,
\[
E_{m,p}^{0}(f):=\dfrac{1}{p^{mn}}\sum_{\overline{x}\in(p\ZZ/p^{m}\ZZ)^{n}}\exp\Big(\frac{2\pi if(x)}{p^{m}}\Big),
\]
was considered by Denef and Sperber in \cite{11}. Under certain conditions on the Newton polyhedron $\Delta$ of $f$, they proved that there exist constants $\sigma, \kappa$, depending only on $\Delta$, and a positive constant $C$, independent of $p$ and $m$, such that for all $m \geq 1$ and almost all $p$, we have
\[
|E_{m,p}^{0}(f)|\leq Cm^{\kappa-1}p^{-m\sigma}.
\]
In \cite{00}, Cluckers proved both conjectures in the case that $f$ is non-degenerate.\\

To generalise these facts, Cluckers and Veys formulated, in \cite{06}, a conjecture related to the log-canonical threshold of an arbitrary polynomial $f$. We will recall the definition of the log-canonical threshold in the next section. They also introduced the following local exponential sum, for each $y\in\ZZ^{n}$:
\[
E_{m,p}^{y}(f):=\dfrac{1}{p^{mn}}\sum_{\overline{x}\in \overline{y}+(p\ZZ/p^{m}\ZZ)^{n}}\exp\Big(\frac{2\pi if(x)}{p^{m}}\Big).
\]
We restate their conjecture here.

\begin{conj}[Cluckers-Veys] There exists  a positive constant $C$ (that may depend on the polynomial $f$), such that for all primes $p$, for all $m\geq 2$ and for all $y\in\ZZ^{n}$, we have
\[
|E_{m,p}(f)|\leq Cm^{n-1}p^{-ma(f)}
\]
and
\[
|E_{m,p}^{y}(f)|\leq Cm^{n-1}p^{-ma_{y,p}(f)}.
\]
Here $a(f)$ is the minimum, over all $b\in\CC$, of the log-canonical thresholds of the polynomials $f(x)-b$. And, for $y\in\ZZ^{n}$, $a_{y,p}(f)$ is the minimum of the log-canonical thresholds at $y'$ of the polynomials $f(x)-f(y')$, where $y'$ runs over $y+p\ZZ_{p}^{n}$.
\end{conj}

In this article, we will prove a special case of the Cluckers-Veys conjecture. More concretely, we will prove the case in which the log-canonical threshold of $f$ is at most a half. We will consider in detail the local sum where $y=0$ and we will afterwards discuss how one can adapt the proofs to obtain uniform upper bounds for $|E^y_{m,p}(f)|$, for $y \in \ZZ^n$, and an upper bound for $|E_{m,p}(f)|$. Our main theorems will be the following.

\begin{maintheorem}\label{*}
Let $n\geq 1$ and let $f\in\ZZ[x_{1},\ldots,x_{n}]$ be a non-constant polynomial with $f(0)=0$. Put $\sigma=\min\big\{c_{0}(f),\frac{1}{2}\big\}$, where $c_0(f)$ is de log-canonical threshold of $f$ at $0$. Then there exists a positive constant $C$, not depending on $p$ and $m$, and a natural number $N$, such that for all $m\geq 1$ and all primes $p>N$, we have
\[
|E_{m,p}^{0}(f)|\leq Cm^{n-1}p^{-m\sigma}.
\]
\end{maintheorem}

\begin{maintheorem}\label{**}
Let $n\geq 1$ and let $f\in\ZZ[x_{1},\ldots,x_{n}]$ be a non-constant polynomial. Put $\sigma=\min\big\{a(f),\frac{1}{2}\big\}$. Then there exists a positive constant $C$, not depending on $p$ and $m$, and a natural number $N$, such that for all $m\geq 2$ and all primes $p>N$, we have
\[
|E_{m,p}(f)|\leq Cm^{n-1}p^{-m\sigma}.
\]
\end{maintheorem}

Remark that by \cite{11}, \cite{12} and \cite{13}, there exists, for each prime $p$, a positive constant $C_{p}$, such that 
\[
|E_{m,p}^{0}(f)|\leq C_{p}m^{n-1}p^{-mc_{0}(f)}
\]
and
\[
|E_{m,p}(f)|\leq C_{p}m^{n-1}p^{-ma(f)},
\]
for all $m\geq 1$. Therefore we know that once the Main Theorems \ref{*} and \ref{**} are proven, they will hold for $N = 1$, possibly after enlarging the constant $C$.

Notice that the homogeneous polynomials $f$ in two variables that are not yet covered by Igusa in \cite{13}, all satisfy that $a(f) \leq \frac{1}{2}$. Hence our results can be seen as a generalisation of a result of Lichtin from \cite{Lich} or of Wright from \cite{Wr}, in which they proved Igusa's conjecture for any homogeneous polynomial of two variables.
%
%

\begin{remark}\label{rem: m=1}
We observe that if $m=1$, then $|E_{1,p}^{0}(f)|=\frac{1}{p^{n}}$. Hence the Main Theorem \ref{*} is trivial for $m=1$ and we only need to prove it for $m\geq 2$.
\end{remark}

We will give two approaches to our main theorems and for the Main Theorem \ref{*} we will give the details of these approaches. The first approach, in Section \ref{sec: model theory}, will make use of model theory, an estimate of the dimension of arc spaces as in \cite{15}, the Cluckers-Loeser motivic integration theory and an estimate of Weil on finite field exponential sums in one variable (see \cite{Weil}). We will also use an idea which is close to the construction of the local Artin map by Lubin-Tate theory. More concretely, we will prove that certain functions do not depend on the choice of a uniformiser in $\QQ_p$, but only on the angular component of the chosen uniformiser. Hence, when varying uniformisers, we obtain orbits of points that have the same image under these functions. In fact, these orbits depend on actions of the group $\mu_{p-1}(\QQ_{p})$, the group of $(p-1)^{\text{th}}$ roots of unity of $\QQ_p$, on the set of uniformisers of $\QQ_{p}$ and on $\QQ_{p}$. The second approach, in Section \ref{sec: geometry}, will use a concrete expression of cohomology, as in \cite{08}. Both of these approaches will use not only Lang-Weil estimates (\cite{Lawe}) for the number of points on varieties over finite fields, but also the theory of Igusa's local zeta functions. In Section \ref{***} we will give some background on log-canonical thresholds, exponential sums and Igusa's local zeta functions. In Section \ref{sec: Main Theorem 1.2}, we will explain how the results from Section \ref{sec: geometry} can be used to prove the Main Theorem \ref{**}. We will end this paper by explaining, in Section \ref{sec: uniform version}, how to obtain uniform upper bounds for all local sums $E^y_{m,p}$. We will do this both from the geometric, as well as from the model theoretic point of view.

We remark that our results can be extended to the ring of integers $\cO_K$ of any number field $K$, but we only work with $\ZZ$ and $\QQ$ to simplify notation.

\section{Log-canonical Thresholds and exponential sums}\label{***}
\subsection{Log-canonical Threshold}
In this section we will recall two possible definitions of the log-canonical threshold of a polynomial $f$.

\begin{defn}\label{01}
Let $f$ be a non-constant polynomial in $n$ variables over an algebraically closed field $K$ of characteristic zero. Let $\pi: Y \rightarrow K^{n}$ be a proper birational morphism on a smooth variety $Y$. For any prime divisor $E$ on $Y$, we denote by $N$ and $\nu-1$ the multiplicities along $E$ of the divisors of $\pi^{*}f$ and $\pi^{*}(dx_{1}\wedge\ldots\wedge dx_{n})$, respectively. For each $x\in Z(f)\subset K^{n}$, the \emph{log-canonical threshold of $f$ at $x$}, denoted by $c_{x}(f)$, is the real number $\inf_{\pi,E}\big\{\frac{\nu}{N}\big\}$, where $\pi$ runs over all proper birational morphisms to $K^n$ and $E$ runs over all prime divisors on $Y$ such that $x \in \pi(E)$. If we fix any embedded resolution $\pi$ of the germ of $f=0$ at $x$, then
\[
c_{x}(f)=\min_{E: x\in\pi(E)}\Big\{\frac{\nu}{N}\Big\}.
\]
Furthermore we always have $c_{x}(f)\leq 1$. We denote by $c(f)=\inf_{x\in Z(f)}c_{x}(f)$ the \emph{log-canonical threshold of $f$}.
\end{defn}

By the following theorem from \cite{15}, which is true for any algebraically closed field $K$ of characteristic zero, there exists a description of the log-canonical threshold in terms of arc spaces and jet spaces.

\begin{thm}[\cite{15}, Corollaries 0.2 and 3.6]\label{02}
Let $f$ be a non-constant polynomial over $K$ in $n$ variables and let $m$ be a natural number. We set
\[
\textnormal{Cont}^{\geq m}(f):=\{x\in K[[t]]^{n}\mid f(x)\equiv 0\textnormal{ mod }t^{m}\}
\]
and 
\[
\textnormal{Cont}_{0}^{\geq m}(f):=\{x\in (tK[[t]])^{n}\mid f(x)\equiv 0\textnormal{ mod }t^{m}\}.
\]
We denote by $\pi_{m}$ the projection from $K[[t]]^{n}$ to $(K[t]/(t^{m}))^{n}$ and we consider the codimensions of $\pi_{m}(\textnormal{Cont}^{\geq m}(f))$ and $\pi_{m}(\textnormal{Cont}_{0}^{\geq m}(f))$ in $(K[t]/(t^{m}))^{n}\cong K^{nm}$. We denote these two values by $\textnormal{codim }\textnormal{Cont}^{\geq m}(f)$ and $\textnormal{codim }\textnormal{Cont}_{0}^{\geq m}(f)$, respectively.
Then the log-canonical threshold of $f$ equals the real number
\[
c(f)=\inf_{m\geq 1}\frac{\textnormal{codim }\textnormal{Cont}^{\geq m}(f)}{m},
\]
and if $f(0)=0$, then the log-canonical threshold of $f$ at 0 equals the real number
\[
c_{0}(f)=\inf_{m\geq 1}\frac{\textnormal{codim }\textnormal{Cont}_{0}^{\geq m}(f)}{m}.
\]
\end{thm}

\subsection{Exponential sum and Igusa local zeta function}\label{03}
In this section we will discuss formulas for the exponential sums $E_{m,p}(f)$ and $E_{m,p}^0(f)$. These formulas can be found in the works of Igusa and Denef on Igusa local zeta functions. Most of the theory in this section comes from \cite{09}. We will just introduce the necessary notation here.\\

Let $K$ be a number field, $\cO$ the ring of algebraic integers of $K$ and $\mathfrak{p}$ any maximal ideal of $\cO$. We denote the completions of $K$ and $\cO$ with respect to $\mathfrak{p}$ by $K_{\mathfrak{p}}$ and $\cO_{\mathfrak{p}}$. Let $q=p^{m}$ be the cardinality of the residue field $k_{\mathfrak{p}}$ of the local ring $\cO_{\mathfrak{p}}$, then $k_{\mathfrak{p}}=\FF_{q}$. For $x\in K_{\mathfrak{p}}$, we denote by $\ord(x)\in\ZZ\cup\{+\infty\}$ the $\mathfrak{p}$-valuation of $x$, $|x|=q^{-\ord(x)}$ and $\text{ac}(x)=x\pi^{-\ord(x)}$, where 
$\pi \in \cO_{\mathfrak{p}}$ is a fixed uniformising parameter for $\mathcal{O}_{\mathfrak{p}}$.

Let $\chi:\cO_{\mathfrak{p}}^{\times}\rightarrow\CC^{\times}$ be a character on the group of units $\cO_{\mathfrak{p}}^{\times}$ of $\cO_{\mathfrak{p}}$, with finite image. By the \emph{order} of such a character we mean the number of elements in its image. The \emph{conductor} $c(\chi)$ of the character is the smallest $c\geq 1$ for which $\chi$ is trivial on $1+\mathfrak{p}^c$. We formally put $\chi(0)=0$. Let $f(x)\in K[x]$ be a polynomial in $n$ variables, $x=(x_{1},\ldots,x_{n})$, with $f\neq 0$, and let $\Phi:K_{\mathfrak{p}}^{n}\rightarrow\CC$ be a \emph{Schwartz-Bruhat function}, i.e., a locally constant function with compact support. We say that $\Phi$ is {\it residual} if $\text{Supp}(\Phi)\subset\cO_{\mathfrak{p}}^{n}$ and $\Phi(x)$ only depends on $x$ mod $\mathfrak{p}$. Thus if $\Phi$ is residual, it induces a function $\overline{\Phi}:k_{\mathfrak{p}}^{n}\rightarrow\CC$. Now we associate to these data \emph{Igusa's local zeta function}
\[
Z_{\Phi}(K_{\mathfrak{p}},\chi,s,f):=\int_{K_{\mathfrak{p}}^{n}}\Phi(x)\chi\big(\text{ac}(f(x))\big)|f(x)|^{s}|dx|.
\]
In \cite{13}, Igusa showed that $Z_{\Phi}(K_{\mathfrak{p}},\chi,s,f)$ is a rational function in $t=q^{-s}$. From now on we will write $Z_{\Phi}(\mathfrak{p},\chi,s,f)$, whenever we have fixed $K$. \\

Let $\Psi$ be the standard additive character on $K_\mathfrak{p}$, i.e. for $z\in K_{\mathfrak{p}}$,
\[
\Psi(z):=\exp(2\pi i\mbox{Tr}_{K_{\mathfrak{p}}/\QQ_{p}}(z)),
\]
where $\mbox{Tr}_{K_{\mathfrak{p}}/\QQ_{p}}$ denotes the trace map. We set
\[
E_{\Phi}(z,\mathfrak{p},f):=\int_{K_{\mathfrak{p}}^{n}}\Phi(x)\Psi(zf(x))|dx|.
\]
Whenever $\Phi=\11_{\cO_{\mathfrak{p}}^n}$ or $\Phi=\11_{(\mathfrak{p}\cO_{\mathfrak{p}})^n}$ and $K$ is fixed, we will simply denote this function by $E_{\mathfrak{p}}(z,f)$ or  $E^{0}_{\mathfrak{p}}(z,f)$, respectively. When $K=\QQ$, $\mathfrak{p}=p\ZZ$, $z=p^{-m}$ and $\Phi=\11_{\ZZ_{p}^n}$ or $\Phi=\11_{(p\ZZ_{p})^n}$ we will simplify notation even more by writing $E_{m,p}(f)$ or $E_{m,p}^{0}(f)$, respectively, and this notation coincide with the notation in Section \ref{00} by an easy calculation.\\

We recall the following proposition from \cite{09}, that relates the exponential sums to Igusa's local zeta functions.

\begin{prop}[\cite{09}, Proposition 1.4.4]\label{04}
Let $u\in\cO_{\mathfrak{p}}^{\times}$ and $m\in\ZZ$. Then $E_{\Phi}(u\pi^{-m},\mathfrak{p},f)$ is equal to

\begin{align*}
Z_{\Phi}(\mathfrak{p},\chi_\textnormal{triv},0,f)&+\mbox{\textnormal{Coeff}}_{t^{m-1}}\Big(\dfrac{(t-q)Z_{\Phi}(\mathfrak{p},\chi_\textnormal{triv},s,f)}{(q-1)(1-t)}\Big)\\
&+\sum_{\chi\neq\chi_\textnormal{triv}}g_{\chi^{-1}}\chi(u)\mbox{\textnormal{Coeff}}_{t^{m-c(\chi)}}\big(Z_{\Phi}(\mathfrak{p},\chi,s,f)\big),
\end{align*}

where $g_{\chi}$ is the Gaussian sum
\[
g_{\chi}=\frac{q^{1-c(\chi)}}{q-1}\sum_{\overline{v}\in(\cO_{\mathfrak{p}}/\mathfrak{p}^{c(\chi)})^{\times}}\chi(v)\Psi(v/\pi^{c(\chi)}).
\]
\end{prop}

Now we will describe a formula for Igusa's local zeta function using resolution of singularities

Let $K$ and $f$ be as above. Put $X=\Spec K[x]$ and $D=\Spec K[x]/(f)$. We take an embedded resolution $(Y,h)$ for $f^{-1}(0)$ over $K$. This means that $Y$ is an integral smooth closed subscheme of projective space over $X$, $h:Y\rightarrow X$ is the natural map, the restriction $h:Y\backslash h^{-1}(D)\rightarrow X\backslash D$ is an isomorphism, and $(h^{-1}(D))_{\text{red}}$ has only normal crossings as subscheme of $Y$. Let $E_{i}, i\in T$, be the irreducible components of $(h^{-1}(D))_{\text{red}}$. For each $i\in T$, let $N_{i}$ be the multiplicity of $E_{i}$ in the divisor of $f\circ h$ on $Y$ and let $\nu_{i}-1$ be the multiplicity of $E_{i}$ in the divisor of $h^{*}(dx_{1}\wedge\ldots\wedge dx_{n})$. The $(N_{i},\nu_{i})_{i \in T}$ are called the \emph{numerical data} of the resolution. For each subset $I\subset T$, we consider the schemes
\[
E_{I}:=\cap_{i\in I} E_{i} \qquad \text{and} \qquad \overset{\circ}{E_{I}}:=E_{I}\backslash\cup_{j\in T\backslash I}E_{j}.
\]
In particular, when $I=\emptyset$ we have $E_{\emptyset}=Y$. We denote the critical locus of $f$ by $C_{f}$.\\

If $Z$ is a closed subscheme of $Y$, we denote the reduction mod $\mathfrak{p}$ of $Z$ by $\overline{Z}$ (see \cite{Shim}). We say that the resolution $(Y,h)$ of $f$ has \emph{good reduction modulo $\mathfrak{p}$} if $\overline{Y}$ and all $\overline{E}_{i}$ are smooth, $\cup_{i\in T}\overline{E}_{i}$ has only normal crossings, and the schemes $\overline{E}_{i}$ and $\overline{E}_{j}$ have no common components whenever $i\neq j$. There exists a finite subset $S$ of $\Spec\cO$, such that for all $\mathfrak{p}\notin S$, we have $f\in\cO_{\mathfrak{p}}[x]$, $f\not\equiv 0\mbox{ mod }\mathfrak{p}$ and the resolution $(Y,h)$ for $f$ has good reduction mod $\mathfrak{p}$ (see \cite{Den}, Theorem 2.4).

Let $\mathfrak{p}\notin S$ and $I\subset T$, then it is easy to prove that $\overline{E}_{I}=\cap_{i\in I}\overline{E}_{i}$. We put $\overset{\circ}{\overline{E}}_{I}:=\overline{E}_{I}\backslash\cup_{j\notin I}\overline{E}_{j}$. Let $a$ be a closed point of $\overline{Y}$ and $T_{a}=\{i\in T| a\in \overline{E}_{i}\}$. In the local ring of $\overline{Y}$ at $a$ we can write 
\[
\overline{f}\circ\overline{h}=\overline{u}\prod_{i\in T_{a}}\overline{g}_{i}^{N_{i}},
\]
where $\overline{u}$ is a unit, $(\overline{g}_{i})_{i\in T_{a}}$ is a part of a regular system of parameters and $N_{i}$ is as above.\\

In two cases, depending on the conductor $c(\chi)$ of the character $\chi$, we will give a more explicit description of Igusa's zeta function $Z_\Phi(\mathfrak{p}, \chi, s, f)$. In the first case we consider a character $\chi$ on $\cO_{\mathfrak{p}}^{\times}$ of order $d$, which is trivial on $1+\mathfrak{p}\cO_{\mathfrak{p}}$, i.e., $c(\chi) = 1$. Then $\chi$ induces a character (denoted also by $\chi$) on $k_{\mathfrak{p}}^{\times}$. We define a map
\[
\Omega_{\chi}:\overline{Y}(k_{\mathfrak{p}})\rightarrow\CC
\]
as follows. Let $a\in\overline{Y}(k_{\mathfrak{p}})$. If $d|N_{i}$ for all $i\in T_{a}$, then we put $\Omega_{\chi}(a)=\chi(\overline{u}(a))$, otherwise we put $\Omega_{\chi}(a)=0$. This definition is independent of the choice of $\overline{g}_{i}$. In the following theorem we recall the formula of Igusa's local zeta function.

\begin{thm}[\cite{08}, Theorem 2.2 or \cite{09}, Theorem 3.4]\label{05}
Let $\chi$ be a character on $\cO_{\mathfrak{p}}^{\times}$ of order $d$, which is trivial on $1+\mathfrak{p}\cO_{\mathfrak{p}}$. Supose that $\mathfrak{p} \notin S$ and that $\Phi$ is residual, then we have
\[
Z_{\Phi}(\mathfrak{p},\chi,s,f)=q^{-n}\sum_{\substack{I\subset T,\\ \forall i\in I: d\mid N_{i}}}c_{I,\Phi,\chi}\prod_{i\in I}\frac{(q-1)q^{-N_{i}s-\nu_{i}}}{1-q^{-N_{i}s-\nu_{i}}},
\]
where 
\[
c_{I,\Phi,\chi}=\underset{a\in\overset{\circ}{\overline{E}}_{I}(k_{\mathfrak{p}})}{\sum}\overline{\Phi}(\overline{h}(a))\Omega_{\chi}(a).
\]  
\end{thm}

If $\Phi=\11_{\cO_{\mathfrak{p}}^n}$ or $\Phi=\11_{(\mathfrak{p}\cO_{\mathfrak{p}})^n}$ we will denote $c_{I,\Phi,\chi}$ by $c_{I,\chi}$ or $c_{I,\chi}^{0}$, respectively.

We note that $c_{I,\Phi,\chi}=0$, if there exists $i\in I$, such that $d\nmid N_{i}$. Therefore the number of characters  $\chi$, for which $c(\chi)=1$ and $c_{I,\Phi,\chi}\neq 0$ for some $I\subset T$, will have an upper bound $M$, which will only depend on the numerical data of $(Y,h)$, hence does not depend on $\text{char}(k_{\mathfrak{p}})$.\\

Now in the second case we consider a character $\chi$ on $\cO_{\mathfrak{p}}^{\times}$, which is non-trivial on $1+\mathfrak{p}\cO_{\mathfrak{p}}$, i.e.\ $c(\chi) >1$. Then we have the following theorem by Denef.

\begin{thm}[\cite{08}, Theorem 2.1 or \cite{09}, Theorem 3.3]\label{06}
Let $\chi$ be a character on $\cO_{\mathfrak{p}}^{\times}$, which is non-trivial on $1+\mathfrak{p}\cO_{\mathfrak{p}}$. Suppose that $\Phi$ is residual, $\mathfrak{p}\notin S$, $N_{i}\notin\mathfrak{p}$ for all $i\in T$, and $C_{\overline{f}}\cap\Supp(\overline{\Phi})\subset\overline{f}^{-1}(0)$. Then $Z_{\Phi}(\mathfrak{p},\chi,s,f)=0$
\end{thm}

As a consequence of these results, one can obtain the following description of the exponential sums $E_\Phi(z, \mathfrak{p},f)$. This result and its proof are very similar to that of Corollary 1.4.5 from \cite{09}.

\begin{cor}\label{07}
Suppose that $\Phi$ is residual, $\mathfrak{p}\notin S$, $N_{i}\notin\mathfrak{p}$ for all $i\in T$, and $C_{\overline{f}}\cap\Supp(\overline{\Phi})\subset\overline{f}^{-1}(0)$. Then $E_{\Phi}(z,\mathfrak{p},f)$ is a finite $\CC$-linear combination of functions of the form $\chi(\textnormal{ac}(z))|z|^{\lambda}(\log_{q}|z|)^{\beta}$ with coefficients independent of $z$, where $\lambda\in\CC$ is a pole of $(q^{s+1}-1)Z_{\Phi}(\mathfrak{p},\chi_{\textnormal{triv}},s,f)$ or of $Z_{\Phi}(\mathfrak{p},\chi,s,f)$, for $\chi \neq \chi_{\textnormal{triv}}$, and $\beta\in\NN$, such that $\beta\leq (\mbox {multiplicity of pole }\lambda)-1$, provided that $|z|$ is big enough. 
\end{cor}

\begin{proof}
It is easy to prove by combining the Theorems \ref{04}, \ref{05} and \ref{06}.
\end{proof}

\section{The first approach by Model theory}\label{sec: model theory}

The first part of this section will contain some background on the theory of motivic integration. For the details we refer to \cite{02} or \cite{05}. In the second part we will use this theory to give our first proof of the Main Theorem \ref{*}.

\subsection{Constructible Motivic Functions}

\subsubsection{The language of Denef-Pas}\label{sec:DP}

Let $K$ be a valued field, with valuation map $\ord:K^{\times}\rightarrow \Gamma_K$ for some additive ordered  group $\Gamma_K$, and let $\cO_K$ be the valuation ring of $K$ with maximal ideal $\cM_K$ and residue field  $k_K$. We denote by $x\rightarrow\overline{x}$ the projection $\cO_K\rightarrow k_K$ modulo $\cM_K$. An angular component map (modulo $\cM_K$) on $K$ is any multiplicative map $\ac:K^{\times}\rightarrow k_K^{\times}$ satisfying $\ac(x)=\overline{x}$ for all $x$ with $\ord(x)=0$. It can be extended to $K$ by putting $\ac(0)=0$.

The \emph{language $\LPas$ of Denef-Pas} is the three-sorted language

\begin{center}
$(\mathcal{L}_{\rm ring},\mathcal{L}_{\rm ring},\mathcal{L}_{\rm oag},\ord,\ac)$
\end{center}
with as sorts:
\begin{itemize}
\item[(i)] a sort $\VF$ for the valued field-sort,
\item[(ii)] a sort $\RF$ for the residue field-sort, and
\item[(iii)] a sort $\VG$ for the value group-sort.
\end{itemize}

The first copy of  $\mathcal{L}_{\rm ring}$ is used for the sort $\VF$, the second copy for $\RF$ and the language $\mathcal{L}_{\rm oag}$, the language $(+,<)$ of ordered abelian groups, is used for $\VG$. Furthermore $\ord$ denotes the valuation map from non-zero elements of $\VF$ to $\VG$, and $\ac$ stands for an angular component map from $\VF$ to $\RF$. 

As usual for first order formulas, $\LPas$-formulas are built up from the $\LPas$-symbols together with variables, the logical connectives $\wedge$ (and), $\vee$ (or), $\neg$ (not), the quantifiers $\exists, \forall$, the equality symbol $=$, and possibly parameters (see \cite{18} for more details).\\

Let us briefly recall the statement of the Denef-Pas theorem on elimination of valued field quantifiers in the language $\LPas$. Denote by $H_{\ac,0}$ the $\LPas$-theory of the above described structures whose valued field is Henselian and whose residue field is of characteristic zero. Then the theory $H_{\ac,0}$ admits elimination of quantifiers in the valued field sort, as stated in the following theorem.

\begin{thm}[Pas, \cite{18}]\label{08}
The theory $H_{\overline{\textnormal{ac}},0}$ admits elimination of quantifiers in the valued field sort. More precisely, every $\mathcal{L}_{\textnormal{DP}}$-formula $\phi(x,\xi,\alpha)$ (without parameters), with $x$ denoting variables in the $\VF$-sort, $\xi$ variables in the $\RF$-sort and $\alpha$ variables in the $\VG$-sort, is $H_{\overline{\textnormal{ac}},0}$-equivalent to a finite disjunction of  formulas of the form
\[
\psi\big(\overline{\textnormal{ac}}f_{1}(x),\ldots,\overline{\textnormal{ac}}f_{k}(x),\xi\big)\wedge\vartheta\big(\ord f_{1}(x),\ldots,\ord f_{k}(x),\alpha\big),
\]
where $\psi$ is an $\cL_{\rm ring}$-formula, $\vartheta$ an $\cL_{\rm oag}$-formula and $f_{1},\ldots,f_{k}$ polynomials in $\mathbb{Z}[X]$.
\end{thm}

This theorem implies the following, useful corollary.

\begin{cor}[\cite{02}, Corollary 2.1.2]\label{QE}Let $(K,k,\Gamma)$ be a model of the theory $H_{\ac,0}$ and $S$ a subring of $K$. Let $T_{S}$ be the set of atomic $\cL_{\textnormal{DP}}\cup S$-sentences and negations of atomic sentences $\varphi$ such that $S\models\varphi$. We take $H_{S}$ to be the union of $H_{\ac,0}$ and $T_{S}$. Then Theorem \ref{08} holds with $H_{\ac,0}$ replaced by $H_{S}$, $\cL_{\textnormal{DP}}$ replaced by $\cL_{\textnormal{DP}}\cup S$, and $\ZZ[X]$ replaced by $S[X]$.
\end{cor}

It is important to remark that by compactness, this theorem and its corollary are still true in the case of $\mathbb{Q}_{p}$ for $p$ sufficiently large.\\

We will need the following notion. Let $k$ be a fixed field of characteristic zero. We denote by $\mathcal{L}_{\text{DP},k}$ the language obtained by adding constant symbols to the language $\LPas$  in the $\VF$, resp.\ $\RF$ sort, for every element of $k((t))$, resp.\ $k$. Then for any field $K$ containing $k$, $(K((t)),K,\ZZ)$ is an $\mathcal{L}_{\text{DP},k}$-structure.

\subsubsection{Constructible motivic functions}
In this section we will recall very quickly the definition of constructible motivic functions. For the details we refer to \cite{02}.

We fix a field $k$ of characteristic zero. Denote by $\text{Field}_{k}$ the category of all fields containing $k$. For any $\cL_{\text{DP},k}$-formula $\phi$, we denote by $h_{\phi}(K)$ the set of points in 
\[
h[m,n,r](K):=K((t))^{m}\times K^{n}\times\ZZ^{r},
\]
which satisfy $\phi$. We call the assignment $K\mapsto h_{\phi}(K)$ a \emph{$k$-definable subassignment} and we define $\text{Def}_k$ to be the category of $k$-definable subassignments. A point $x$ of $X \in \text{Def}_k$ is a tuple $x=(x_{0},K)$ where $x_{0}\in X(K)$ and $K\in \text{Field}_{k}$. In general, for $S\in \text{Def}_{k}$ we define the category $\text{Def}_{S}$ of definable subassigments $X$ with a definable map $X\rightarrow S$. We denote $\text{RDef}_{S}$  for the category of definable subassignments of $S\times h[0,n,0]$ where $n\in\NN $. We recall that the Grothendieck semigroup $SK_{0}(\RDef_{S})$ is the quotient of the free abelian semigroup over symbols $[Y\rightarrow S]$, with $Y\rightarrow S$ in $\RDef_{S}$, by the relations 
\begin{itemize}
\item[(1)] $[\emptyset\rightarrow S]=0$;
\item[(2)] $[Y\rightarrow S]=[Y'\rightarrow S]$, if $Y\rightarrow S$ is isomorphic to $Y'\rightarrow S$;
\item[(3)] $[(Y\cup Y')\rightarrow S]+[(Y\cap Y')\rightarrow S]=[Y\rightarrow S]+[Y'\rightarrow S]$, for $Y$ and $Y'$ definable subassignments of some $S[0,n,0] = S \times h[0,n,0] \rightarrow S$.
\end{itemize}
Similarly, we recall that the Grothendieck group $K_{0}(\RDef_{S})$ is the quotient of the free abelian group over the symbols $[Y\rightarrow S]$, with $Y\rightarrow S$ in $\RDef_{S}$, by the relations (2) and (3). The Cartesian fiber product over $S$ induces a natural semi-ring (resp.\ ring) structure on $SK_{0}(\RDef_{S})$ (resp.\ $K_{0}(\RDef_{S})$) by  setting
\begin{center}
$[Y\rightarrow S]\times [Y'\rightarrow S]=[Y \times_S Y'\rightarrow S]$.
\end{center}
We consider a formal symbol $\mathbb{L}$ and the ring
\[
\mathbb{A}:=\mathbb{Z}\Big[\mathbb{L},\mathbb{L}^{-1},\Big(\frac{1}{1-\mathbb{L}^{-{i}}}\Big)_{i>0}\Big].
\]
For every real number $q>1$, there is a unique morphism of rings $\vartheta_{q}:\mathbb{A}\rightarrow \mathbb{R}$ mapping $\mathbb{L}$ to $q$, and it is obvious that $\vartheta_{q}$ is injective for  $q$ transcendental. We define a partial ordering on $\mathbb{A}$ by setting $a\geq b$ if, for every real number $q>1$, $\vartheta_{q}(a)\geq \vartheta_{q}(b)$. Furthermore we denote by $\mathbb{A}_{+}$ the set $\lbrace a\in \mathbb{A}|a\geq 0\rbrace$.

\begin{defn}\label{09}
Let $S$ be a definable subassignment in $\text{Def}_{k}$ and denote by $\lvert S \rvert$ its set of points. The ring $\mathcal{P}(S)$ of \emph{constructible Presburger functions on $S$} is defined as the subring of the ring of functions $\lvert S \rvert \rightarrow \mathbb{A}$, generated by
\begin{itemize}
\item the constant functions $\lvert S \rvert \rightarrow \mathbb{A}$;
\item the functions $\widehat{\alpha}: \lvert S \rvert \rightarrow\mathbb{Z}$ that correspond to a definable morphism $\alpha:S\rightarrow h[0,0,1]$;
\item the functions $\mathbb{L}^{\widehat{\beta}}:\lvert S \rvert \rightarrow \mathbb{A}$ that correspond to a definable morphism $\beta: S\rightarrow h[0,0,1]$.
\end{itemize}
We denote by $\mathcal{P}_{+}(S)$ the semiring of funtions in $\mathcal{P}(S)$ with values in $\mathbb{A}_{+}$.
\end{defn}

\begin{defn}\label{10}
Let $Z$ be in $\text{Def}_{k}$. For $Y$ a definable subassignment of $Z$, we denote by \textbf{1}$_{Y}$ the function in $\mathcal{P}(Z)$ with value $1$ on $|Y|$ and $0$ on $|Z\backslash Y|$. We denote by $\mathcal{P}^{0}_{Z}$ (resp.\ $\mathcal{P}^{0}_{+}(Z)$) the subring (resp.\ subsemiring) of $\mathcal{P}(Z)$ (resp.\ $\mathcal{P}_{+}(Z)$) generated by the functions \textbf{1}$_{Y}$, for all definable subassignments $Y$ of $Z$, and by the constant function $\mathbb{L}-1$. Notice that we have a canonical ring morphism $\mathcal{P}^{0}(Z)\rightarrow K_{0}(\text{RDef}_{Z})$ (resp.\ semiring morphism $\mathcal{P}^{0}_{+}(Z)\rightarrow SK_{0}(\text{RDef}_{Z})$) sending \textbf{1}$_{Y}$ to the class of the inclusion morphism $[i:Y\rightarrow Z]$ and $\mathbb{L}-1$ to $\mathbb{L}_{Z}-1$. By $\mathbb{L}_{Z}$ we mean the class of the element $[Z\times h[0,1,0] \to Z]$ in $K_{0}(\text{RDef}_{Z})$ (resp.\ $SK_{0}(\text{RDef}_{Z})$).
\end{defn}

\begin{defn}\label{10.1}
We say that a function $\varphi\in\mathcal{P}(S\times\mathbb{Z}^{r})$ is \emph{$S$-integrable}, if for every $s\in S$, the family $(\varphi(s,i))_{i\in\mathbb{Z}^{r}}$ is summable. We denote by $I_{S}\mathcal{P}(S\times\mathbb{Z}^{r})$ the $\mathcal{P}(S)$-module of $S$-integrable functions.

Now we define the semiring $\mathcal{C}_{+}(Z)$ of \emph{positive constructible motivic functions on $Z$} as 

\begin{center}
$\mathcal{C}_{+}(Z)=SK_{0}(\text{RDef}_{Z})\otimes_{\mathcal{P}^{0}_{+}(Z)}\mathcal{P}_{+}(Z)$
\end{center}
and the ring $\mathcal{C}(Z)$ of \emph{constructible motivic functions on $Z$} as
\begin{center}
$\mathcal{C}(Z)=K_{0}(\text{RDef}_{Z})\otimes_{\mathcal{P}^{0}(Z)}\mathcal{P}(Z)$.
\end{center}
\end{defn}

Let $Z$ be a subassignment of $h[m,n,r]$. We denote by $\dim Z$ the dimension of Zariski closure of $p(Z)$ for $p$ the projection $h[m,n,r]\rightarrow h[m,0,0]$. For a natural number $d$, we denote by $\mathcal{C}^{\leq d}$ the ideal of $\mathcal{C}(Z)$ generated by all elements of the form $\textbf{1}_{Y}$ with $Y$ a subassignment of $Z$ such that $\dim Y\leq d$. We set  $\mathcal{C}^{d}=\cC^{\leq d}\diagup\cC^{\leq d-1}$ and $C(Z)=\oplus_{d\geq 0}\cC^{d}$. 

For each $Y$ in $\text{Def}_{S}$ we can define a graded subgroup $I_{S}(Y)$ of $C(Y)$, as in \cite{02}, together with a map  $f_{!}: I_{S}(Y)\rightarrow I_{S}(Z)$, for any map $f:Y\rightarrow Z$ in $\text{Def}_{S}$. When $S=h[0,0,0]$ and $f:Y\rightarrow h[0,0,0]$, the map  $f_{!}$ is exactly the same as taking the integral over $Y$.

\subsubsection{The language $\mathcal{L}_{\mathcal{O}}$}\label{11}

Now we suppose that $K$ is a number field with $\mathcal{O}$ its ring of integers. We denote by $\mathcal{F}_{\mathcal{O}}$ the set of all non-archimedean local fields over $\mathcal{O}$, which is endowed the structure of an $\mathcal{O}$-algebra. For $N\in \NN$ we denote by $\mathcal{F}_{\mathcal{O},N}$ the set of all local fields in $\mathcal{F}_{\mathcal{O}}$ with residue  field of characteristic at least $N$. The language $\mathcal{L}_{\mathcal{O}}$ is obtained from the language $\mathcal{L}_{\text{DP},K}$ by restricting the constant symbols to $\mathcal{O}[[t]]$ for the valued field sort and to $\mathcal{O}$ for the residue field sort.

Let $F\in \mathcal{F}_{\mathcal{O}}$, we write $k_{F}$ for its residue field, $q_{F}$ for the number of elements in $k_F$,  $\mathcal{O}_{F}$ for its valuation ring and $\mathcal{M}_{F}$ for its maximal ideal. For each choice of a uniformising element $\varpi_{F}$ of $\mathcal{O}_{F}$, there is a unique map $\ac_{\varpi_F}: F^{\times}\rightarrow k_{F}^{\times}$, which extends  the map $\mathcal{O}_{F}^{\times}\rightarrow k_{F}^{\times}$ and sends $\varpi_{F}$ to $1$. Then $(F,k_{F},\ZZ)$ has an $\mathcal{L}_\text{DP}$-structure with respect to $\varpi_{F}$. Moreover $F$ can be equipped with the structure of an $\mathcal{O}[[t]]$-algebra via the morphism:
\begin{align*}
\lambda_{\varpi_{F}}:\mathcal{O}[[t]]&\rightarrow F;\\
\sum_{i\geq 0}a_{i}t^{i}&\mapsto\sum_{i\geq 0}a_{i}\varpi_{F}^{i}.
\end{align*}
By intepreting $a\in \mathcal{O}[[t]]$ as $\lambda_{\varpi_{F}}(a)$, an $\mathcal{L}_{\mathcal{O}}$-formula $\phi$ defines, for each $F\in\mathcal{F}_{\mathcal{O}}$, a definable subset $\phi(F)$ of $F^{m}\times k_{F}^{n}\times\ZZ^{r}$ for some $m,n,r\in \NN$. If we have two $\mathcal{L}_{\mathcal{O}}$-formulas $\phi_{1},\phi_{2}$ which define the same subassignment of $h[m,n,r]$, then, by compactness, $\phi_{1}(F)=\phi_{2}(F)$, for all $F\in\mathcal{F}_{\cO,N}$, for some large enough $N\in \NN$, which does not depend on the choice of a uniformising element.\\

If a definable subassignment is defined in the language $\cL_{\cO}$, then we say that it belongs to $\text{Def}_{\cL_{\cO}}$. In the same way we also say that a constructible function $\theta$ belongs to $\cC(X,\cL_{\cO})$.

If $X\in \text{Def}_{\cL_{\cO}}$, then $X$ is defined by a formula $\phi$ in $\cL_{\cO}$. By the above discussion we can define $X_{F}=\phi(F)$, for any $F\in \cF_{\cO}$. Also if $f:Y\rightarrow Z$ in $\text{Def}_{\cL_{\cO}}$, then we can define a map $f_{F}:Y_{F}\rightarrow Z_{F}$, for any $F\in \cF_{\cO}$.

Now we will explain how to interprete a constructible function $\theta\in \cC(X,\cL_{\cO})$ in a field $F\in\cF_{\cO}$.
If $\theta\in\cP(X)$ we will replace $\LL$ by $q_{F}$ and a definable function $\alpha: X\rightarrow h[0,0,1]$ by a function $\alpha_{F}:X_{F}\rightarrow \ZZ$.
If $\theta\in K_{0}(\text{RDef}_{X,\cL_{\cO}})$ is of the form $[Y\overset{\pi}{\rightarrow} X]$ with $\pi:Y\rightarrow X$ defined by an $\cL_{\cO}$-formula, then we interpret $\theta$ by setting, for all $x \in X_F$,
\[
\theta_{F}(x):=\#(\pi^{-1}(x)).
\]
Notice that these interpretations can depend on the choice of formulas.

\subsubsection{Cell decomposition}

The structure of the sets appearing in a definable subassignment, can be better understood by decomposing the subassignment into `cells'.

\begin{defn} \textbf{Cells}.\label{12} Let $S$ be in Def$_K$ and  $C$ a definable subassignment of $S$. Let $\alpha, \xi, c$ be definable morphisms $\alpha: C\rightarrow h[0,0,1], \xi:C\rightarrow h[0,1,0]$ and $c: C\rightarrow h[1,0,0]$. The \emph{$1$-cell} $Z_{C,\alpha,\xi,c}$ with \emph{basis} $C$, \emph{order} $\alpha$, \emph{center} $c$, and \emph{angular component} $\xi$, is the definable subassignment of $S[1,0,0]$, defined by the formula
\begin{center}
$y\in C \wedge \ord(z-c(y))=\alpha(y) \wedge \overline{\text{ac}}(z-c(y))=\xi(y),$
\end{center}
where $y$ belongs to $S$ and $z$ to $h[1,0,0]$. Similarly the \emph{$0$-cell} $Z_{C,c}$ with \emph{basis} $C$ and \emph{center} $c$, is the definable subassignment of $S[1,0,0]$, defined by the formula
\begin{center}
$y\in C \wedge z=c(y)$.
\end{center}
A definable subassignment $Z$ of $S[1,0,0]$ will be called a \emph{$1$-cell}, resp.\ a \emph{$0$-cell}, if there exists a definable isomorphism 
\begin{center}
$\lambda: Z\rightarrow Z_{C}=Z_{C,\alpha,\xi,c}\subset S[1,s,r]$,
\end{center}
resp.\ a definable isomorphism
\begin{center}
$\lambda: Z\rightarrow Z_{C}=Z_{C,c}\subset S[1,s,0]$,
\end{center}
for some $r,s\geq 0$, some basis $C\subset S[0,s,r]$, resp.\ $S[0,s,0]$, and some $1$-cell $Z_{C,\alpha,\xi,c}$, resp.\ $0$-cell $Z_{C,c}$, such that the morphism $\pi\circ \lambda$, with $\pi$ the projection on the $S[1,0,0]$-factor, is the identity on $Z$. The data $(\lambda, Z_{C,\alpha,\xi,c})$, resp.\ $(\lambda, Z_{C,c})$, will be called a \emph{presentation of the cell $Z$} and denoted for short by $(\lambda, Z_{C})$.
\end{defn}

\begin{thm}[\cite{02}, Thm 7.2.1]\label{13} Suppose that $K$ is a field of characteristic $0$. Let $X$ be a definable subassignment of $S[1,0,0]$ with $S$ in $\textnormal{Def}_{K}$.
\begin{itemize}
\item[(1)] The subassignment $X$ is a finite disjoint union of cells.
\item[(2)] For every $\varphi\in \mathcal{C}(X)$, there exists a finite partition of $X$ into cells $Z_{i}$ with presentation $(\lambda_{i},Z_{C_{i}})$, such that $\varphi|_{Z_{i}}=\lambda_{i}^{*}p_{i}^{*}(\psi_{i})$, with $\psi_{i}\in \mathcal{C}(C_{i})$ and $p_{i}:Z_{C_{i}}\rightarrow C_{i}$ the projection. Similar statements hold for $\varphi$ in $\mathcal{C}_{+}(C)$, in $\mathcal{P}(X)$, in $\mathcal{P}_{+}(X)$, in $K_{0}(\textnormal{RDef}_{Z})$ and in $SK_{0}(\textnormal{RDef}_{Z})$.
\end{itemize}
\end{thm}

\begin{cor}\label{cell}
Theorem \ref{13} still holds, if we replace $\Def_{K}$ by $\Def_{\cL_{\cO}}$.
\end{cor}
\begin{proof}
The proof is the same as the proof of Theorem \ref{13}, but we replace $\cL_{\textnormal{DP,K}}$ by $\cL_{\cO}\subset \cL_{\textnormal{DP,K}}$.
\end{proof}

\subsection{Proof of the main theorem}

We will give a proof of the Main Theorem \ref{*} by splitting the exponential sum $E_{m,p}^0(f)$ into three subsums.
\begin{align*}
&E_{p,m}^0(f) = \frac{1}{p^{nm}} \sum_{\substack{\overline{x} \in (p\ZZ/p^m\ZZ)^n,\\ \ord_p(f(x)) \leq m-2}} \exp\left(\frac{2 \pi i}{p^m} f(x)\right) +\\
&\frac{1}{p^{nm}}\sum_{\substack{\overline{x} \in (p\ZZ/p^m\ZZ)^n,\\ \ord_p(f(x)) = m-1}} \exp\left(\frac{2 \pi i}{p^m} f(x)\right) + \frac{1}{p^{nm}}\sum_{\substack{\overline{x} \in (p\ZZ/p^m\ZZ)^n,\\ \ord_p(f(x)) \geq m}} \exp\left(\frac{2 \pi i}{p^m} f(x)\right).
\end{align*}
In three different lemmas we will analyse each of these sums.

For the first subsum we will introduce a constructible function $G$, that expresses, for a certain input $z \in \ZZ_p$ with $\ord_p(z) \leq m-2$, how many $x \in (p\ZZ_p)^n$ are mapped close to $z$ by $f$. We will apply the Cell Decomposition Theorem to $G$ and with some further techniques like eliminiation of quantifiers, we will show that certain values $z$ of $f$ occur equally often. In the exponential sum these values will cancel out.

\begin{lem}\label{14}
Let $f\in\ZZ[x_{1},\ldots,x_{n}]$ be a non-constant polynomial such that $f(0)=0$. There exists $N\in\NN$ such that, for all $m\geq 1$ and all prime numbers $p>N$, we have
\[
\sum_{\substack{\overline{x}\in(p\ZZ/p^{m}\ZZ)^{n},\\ \ord_{p}(f(x))\leq m-2}}\exp\left(\frac{2\pi if(x)}{p^{m}}\right)=0.
\]
\end{lem}

\begin{proof} The statement is obvious when $m=1$ or $m=2$, so we can assume that $m>2$. Let $\phi$ be the $\cL_{\ZZ}$-formula given by 
\[
\phi(x_{1},\ldots,x_{n},z,m)=\bigwedge_{i=1}^{n}(\ord(x_{i})\geq 1)\wedge(\ord(z)\leq m-2)\wedge(\ord(z-f(x_{1},\ldots,x_{n}))\geq m),
\]
where $x_{i},z$ are in the valued field-sort and $m$ is in the value group-sort. To shorten notation we set $x=(x_{1},\ldots,x_{n})$. For each prime $p$, we fix a uniformiser $\varpi_{p}$ of $\QQ_{p}$, then $\phi$ defines, for each $p$, a definable set $X_{p}\subset(p\ZZ_{p})^{n}\times\ZZ_{p}\times\ZZ$. More precisely, we have
\[
X_{p}=\{(x,z,m)\in (p\ZZ_{p})^{n}\times\ZZ_{p}\times\ZZ  \mid \ord_{p}(f(x)-z)\geq m, \ord_p(z)\leq m-2\}.
\]
It is obvious that $X_{p}$ does not depend on $\varpi_{p}$.

We denote by $X\subset h[n+1,0,1]$ the definable subassignment defined by $\phi$. Let $F:=\textbf{1}_{X}\in I_{h[0,0,1]}(h[n+1,0,1])$ and $\pi$ the projection from $h[n+1,0,1]$ to $h[1,0,1]$. Then we have $G:=\pi_{!}(F)\in I_{h[0,0,1]}(h[1,0,1])$. For each prime $p$ and each uniformiser $\varpi_{p}$ of $\QQ_{p}$, there exist the following interpretations of $F$ and $G$ in $\QQ_{p}$:
\[
F_{\varpi_{p}}=\textbf{1}_{X_{p}}
\]
and 
\[
G_{\varpi_{p}}(z,m)=\int_{X_{p,z,m}}|dx|=p^{-mn}\#\{\overline{x}\in (p\ZZ/p^{m}\ZZ)^{n} \mid f(x)\equiv z\mbox{ mod }p^{m}\},
\]
if $\ord_{p}(z)\leq m-2$, where $X_{p,z,m}$ is the fiber of $X_{p}$ over $(z,m)$, and
\[
G_{\varpi_{p}}(z,m)=0,
\]
if $\ord_p(z)\geq m-1$. We can  see that both $F_{\varpi_{p}}(x,z,m)$ and 
$G_{\varpi_{p}}(z,m) $ do not depend on $\varpi_{p}$.\\

Now we use Corollary \ref{cell} for $G\in I_{h[0,0,1]}(h[1,0,1])$. This means that there exists a finite partition of $h[1,0,1]$ into cells $Z_{i}$ (for $i$ in some finite set $I$) with presentation $(\lambda_{i},Z_{C_{i},\alpha_{i},\xi_{i},c_{i}})$, such that $G|_{Z_{i}} = \lambda_{i}^{*} p_{i}^{*}(G_{i})$ with $G_{i}\in \cC(C_{i})$ and $p_{i}:Z_{C_{i},\alpha_{i},\xi_{i},c_{i}}\rightarrow C_{i}$ the projection. Note that $C_{i}\subset h[0,r_{i},s_{i}+1]$ for some $r_{i},s_{i}\in\NN$. We denote by $\theta_{i}(z,\eta,\gamma,m)$ the $\cL_{\ZZ}$-formula defining $c_{i}$, where $z \in h[1,0,0], \eta\in h[0,r_{i},0],\gamma\in h[0,0,s_{i}]$ and $m\in h[0,0,1]$. By elimination of quantifiers (Corollary \ref{QE}), there exist  polynomials $f_{1},\ldots,f_{r}$ in one variable $z$ with coefficients in $\ZZ[[t]]$, such that $\theta_{i}(z, \eta, \gamma,m)$ is equivalent to the formula
\[
\bigvee_{j}\Big(\zeta_{ij}\big(\ac f_{1}(z),\ldots,\ac f_{r}(z),\eta\big)\wedge\nu_{ij}\big(\ord f_{1}(z),\ldots,\ord f_{r}(z)\big),\gamma,m\Big),
\]
where  $\zeta_{ij}$ is an $\cL_{\rm ring}$-formula and $\nu_{ij}$ an $\cL_{\rm oag}$-formula. Since $c_{i}$ is a function, we know that, for each $(\eta,\gamma,m)\in C_{i}$, there exists a unique $z=c_{i}(\eta,\gamma,m)$ such that $\theta_{i}(z,\eta,\gamma,m)$ is true. We claim now that there exists $1\leq l\leq r$ such that $f_{l}(z)=0$. Indeed, if $f_{l}(z)\neq 0$, for all $l$, then there exists a small open neighborhood $V$ of $z$ and there exists an index $j$, such that, for all $y\in V$, $(y,\eta, \gamma,m)$ will satisfy the formulas $\zeta_{ij},\eta_{ij}$. Since this would contradict the uniqueness of $z$, we must have that $f_{l}(z)=0$ for some $l$. We deduce that $A:=\{c_{i}(\eta,\gamma, m) \in h[1,0,0]\mid i\in I,(\eta,\gamma,m)\in C_{i}\}\subset\cup_{j=1}^{r}Z(f_{j})$.\\

From the definition of $G_{i}$ we see that, if we fix $(\eta,\gamma,m)\in C_{i}$, then $G(\cdot,m)$ will be constant on the ball
\[
{\{y \in h[1,0,0] \mid \ac(y-c_{i}(\eta,\gamma,m))=\xi_{i}(\eta,\gamma,m),\ord(y-c_{i}(\eta,\gamma,m))=\alpha_{i}(\eta,\gamma,m)\}}.
\]
Now, for each $m>2$, we set
\[
B_{m}:=A\cap\{z\in h[1,0,0] \mid m-2\geq \ord(z)\geq 1\}
\]
and
\[
U_{m}:=\{y\in h[1,0,0]\mid\ord(z-y) < m-1, \forall z\in B_{m}\}.
\]
So $U_{m}$ will be a union of balls of radius $m-1$. Because $f(0)=0$, we can see that $G(\cdot,m)$ will be zero on the set $\{z\in h[1,0,0] \mid \ord(z)\leq 0\}$, if $m>2$.

\begin{claim}\label{claim1}
If $m>2$, $\ord(z)\geq 1$ and $z\in U_{m}$, then $G(\cdot,m)$ will be constant on the ball $B(z,m-1)$ (the ball with center $z$ and radius $m-1$).
\end{claim}
From the cell decomposition of $h[1,0,1]$, we know that there exist $i\in I$ and $(\eta,\gamma) \in h[0,r_i,s_i]$, such that $(z,\eta,\gamma, m) \in Z_{C_i, \alpha_i, \xi_i, c_i}$. Hence $(\eta,\gamma,m)\in C_{i}$ and $z$ belongs to the ball
\[
{B = \{y \in h[1,0,0]\mid \ac(y-c_{i}(\eta,\gamma,m))=\xi_{i}(\eta,\gamma,m),\ord(y-c_{i}(\eta,\gamma,m))=\alpha_{i}(\eta,\gamma,m)\}}.
\]
We will distinguish three cases, depending on the value of $c_i(\eta,\gamma, m)$. First of all, if $c_{i}(\eta,\gamma,m)\in B_{m}$, then we see that $\alpha_{i}(\eta,\gamma,m)  = \ord(z - c_i(\eta,\gamma, m)) <m-1$. Therefore the ball $B$ will contain the ball $B(z,m-1)$, thus $G(\cdot,m)$ will be constant on $B(z,m-1)$. Second of all, if  $\ord(c_{i}(\eta,\gamma,m))\leq 0$, and since $\ord(z)\geq 1$, we have $\alpha_{i}(\eta,\gamma,m)\leq 0< m-1$ so we have the same situation as above. Thirdly, if $\ord(c_{i}(\eta,\gamma,m))\geq m-1$, then the case $\alpha_i(\eta,\gamma,m) < m-1$ has already been treated above. Hence we can assume that $\alpha_i(\eta,\gamma,m) \geq m-1$, in which case we have $B(z,m-1)=B(0,m-1)$. By definition of $G$ we have $G(\cdot,m)|_{B(0,m-1)}=0$. This proves the claim.\\

Now there exists $N_{0}\in \NN$, independent of $m>2$, for which we can interpret all of the above discussion in $\QQ_{p}$, with any choice of uniformiser $\varpi_{p}\in\ZZ_{p}$ and for any $p>N_{0}$, by applying the map $\lambda_{\varpi_p}$ to the coefficients of the polynomials $f_1, \ldots, f_r$. Because $U_{m,\varpi_{p}}$ is an $\{m,\varpi_{p}\}$-definable set in the language $\cL_{\text{DP}}$, it can vary when changing $\varpi_p$. This suggests us to set $\cU_{m,p}:=\cup_{\varpi_{p}}U_{m,\varpi_{p}}$ with $\varpi_{p}$ running over the set of all uniformisers of $\QQ_{p}$. Then $\cU_{m,p}$ is  given by an $\cL_{\text{DP}}$-formula.

\begin{claim}
There exists $N \in \NN$, such that $\cU_{m,p}=\QQ_{p}$, for all $m>2$ and for all $p>N$.
\end{claim}
From the definition of $U_{m,\varpi_{p}}$ we see that $\cV_{m,p}:=\QQ_{p}\setminus \cU_{m,p}$ is a union of $d_{m,p}$ balls of radius $m-1$, contained in $p\ZZ_{p}$, where $d_{m,p}\leq\sum_{j=1}^{r}\deg f_{j}$. Moreover, $\cV_{m,p}$ will given by a $\cL_{\text{DP}}$-formula. We use elimination of quantifiers (Theorem \ref{08}) for the formula defining $\cV_{m,p}$. Hence there exist polynomials $q_{1},\ldots,q_{\tilde{r}}$ of one variable $z$ with coefficients in $\ZZ$ and formulas $\varphi_{j}$ in $\cL_{\rm ring}$ and $\nu_{j}$ in $\cL_{\rm oag}$, for $1\leq j\leq s$, such that
\[
z\in\cV_{m,p}\Leftrightarrow \bigvee_{j=1}^{s}\varphi_{j}\big(\ac_{\varpi_{p}}(q_{1}(z)),\ldots,\ac_{\varpi_{p}}(q_{\tilde{r}}(z))\big)\wedge\nu_{j}\big(\ord_{p}(q_{1}(z)),\ldots,\ord_{p}(q_{\tilde{r}}(z)),m\big),
\]
for any $p >N_0$ (after enlarging $N_0$ if necessary) and any uniformiser $\varpi_{p}$.\\

 We note that if $z\in\cV_{m,p}$, then $\ord_{p}(z)\geq 1$. Since $q_{i}$ has coefficients in $\ZZ$, we can assume, by possibly enlarging $N_0$, that $\ac_{\varpi_{p}}(q_{i}(z))$ only depends on $\ac_{\varpi_{p}}(z)$ and $\ord_{p}(q_{i}(z))$ only depends on $\ord_{p}(z)$, for any $p>N_{0}$. This follows from the $t$-adic version of this statement by a compactness argument. So if $z_{1}$ and $z_{2}$ satisfy that
\begin{itemize}
\item $\ord_{p}(z_{1})=\ord_{p}(z_{2})\geq 1$,
\item there exist two uniformisers $\varpi_{1,p}$ and $\varpi_{2,p}$, such that $\ac_{\varpi_{1,p}}(z_{1})=\ac_{\varpi_{2,p}}(z_{2})$,
\end{itemize}
then we see that $z_{1}\in \cV_{m,p}$ if and only if $z_{2}\in\cV_{m,p}$. It implies that $\overline{\cV}_{m,p}:= \ac_{\varpi_{p}}(\cV_{m,p})$ does not depend on $\varpi_{p}$, for any $p > N_0$. In particular, since $B(0,m-1)\nsubseteq \cV_{m,p}$, we see that the number of elements in $\overline{\cV}_{m,p}$ is at most $\sum_{j=1}^{r}\deg f_{j}$, for all $p>N_{0}$.\\

In what follows we will show that if $\cV_{m,p}$ were not empty, then the set $\overline{\cV}_{m,p}$ would grow with $p$. This will give the desired contradiction. We set
\[
B_{\infty}=A\cap\{z\in h[1,0,0] \mid \infty>\ord(z)\geq 1\}\subset\cup_{j=1}^{r}Z(f_{j}),
\]
thus $B_{\infty}$ is a finite set with $0\notin B_{\infty}$ and  $B_{m}\subset B_{\infty}$ for all $m>2$. Looking at the order of the coefficients of $f_{j}$ we see that there exists $M\in\NN$ such that $\ord_{p}(z)\leq M$ for all $z\in Z(f_{j,\varpi_{p}})\backslash\{0\}$, for all $1\leq j\leq r$, for all $p>N_{0}$ and for all uniformiser $\varpi_{p}$. So $\ord_{p}(z)\leq M$ for all $z\in B_{\infty,{\varpi_p}}$, for all $\varpi_{p}$. It follows that $\ord_{p}(z)\leq M$ for all $z\in\cV_{m,p}$, for all $m>2$ and $p>N_{0}$. Indeed, since $B(0,m-1)\nsubseteq \cV_{m,p}$ we have $\ord_{p}(z)<m-1$ for all $z\in\cV_{m,p}$, so it is true if $m-1\leq M$. On the other hand, if $m-1>M$, then for each $z\in\cV_{m,p}$ and each uniformiser $\varpi_{p}$, there exists  $z_{0}\in B_{\infty,{\varpi_p}}$ such that $\ord_{p}(z-z_{0})\geq m-1>M\geq \ord_{p}(z_{0})$, thus $\ord_{p}(z) = \ord_p(z_0) \leq M$. Now put $N :=\max\{N_{0},1+M\sum_{j=1}^{r}\deg f_{j}\}$. Suppose for a contradiction, that for some $p > N$, there exists $z \in \cV_{p,m}$. Then $\ac_{\varpi_{p}}(z)\in \overline{\cV}_{m,p}$, for every uniformiser $\varpi_{p}$, and so $\{\ac_{\varpi_{p}}(z) \mid \ord_{p}(\varpi_{p})=1\}\subset \overline{\cV}_{m,p}$. Suppose that $\ac_{p}(\varpi_{p})=u$, then $u^{\ord_{p}(z)}\ac_{\varpi_{p}}(z)=\ac_{p}(z)$, so we have $\{\ac_{\varpi_{p}}(z) \mid \ord_{p}(\varpi_{p})=1\}=\{u^{-\ord_{p}(z)}\ac_{p}(z) \mid u\in\FF_{p}^{\times}\}$. Therefore $\#\{u^{-\ord_{p}(z)}\ac_{p}(z) \mid u\in\FF_{p}^{\times}\}\leq \sum_{j=1}^{r}\deg f_{j}$. But
\[
\#\{u^{-\ord_{p}(z)}\ac_{p}(z) \mid u\in\FF_{p}^{\times}\}=\dfrac{p-1}{\gcd(\ord_{p}(z),p-1)}\geq\dfrac{p-1}{\ord_{p}(z)}\geq\frac{p-1}{M},
\]
 where $\gcd(a,b)$ is the greatest common divisor of $a$ and $b$. Then we have $p-1\leq M\sum_{j=1}^{r}\deg f_{j}\leq N-1$. This is a contradiction, since $p>N$. So this proves the claim.\\

We know from Claim \ref{claim1} that if $m>2$ and $z\in U_{m,\varpi_{p}}$ such that $1 \leq \ord_{p}(z)\leq m-2$, then $G_{\varpi_{p}}(.,m)$ will be constant on the ball $B(z,m-1)$. Thus we have
\[
\#\{\overline{x}\in (p\ZZ/p^{m}\ZZ)^{n} \mid f(x)\equiv y\text{ mod }p^{m}\}=\#\{\overline{x}\in (p\ZZ/p^{m}\ZZ)^{n} \mid f(x)\equiv z\text{ mod }p^{m}\}
\]
for all $\overline{y}\in p\ZZ/p^{m}\ZZ$ with  $y\equiv z\text{ mod }p^{m-1}$. Hence 
\[
\sum_{\substack{\overline{x}\in (p\ZZ/p^{m}\ZZ)^{n},\\f(x)\equiv z \text{ mod } p^{m-1}}}p^{-mn}\exp\Big(\frac{2\pi if(x)}{p^{m}}\Big)= G(z,m) \cdot \sum_{\substack{\overline{y} \in p\ZZ/p^m\ZZ,\\ y \equiv z \text{ mod } p^{m-1}}} \exp\Big(\frac{2 \pi i y}{p^m}\Big) = 0.
\]
This implies that
\[
\sum_{\substack{\overline{x}\in (p\ZZ/p^{m}\ZZ)^{n},\\\overline{f(x)}\in \overline{\cU}_{m,p}}}p^{-mn}\exp\Big(\frac{2\pi if(x)}{p^{m}}\Big)=0,
\]
where $\overline{\cU}_{m,p}:=\{\overline{z} \in p\ZZ/p^{m-1}\ZZ \mid z\in \cU_{m,p},m-2\geq \ord_p(z)\geq 1\}$. For all $m>2$ and $p>N$, we have $\cU_{m,p}=\QQ_{p}$, so $\overline{\cU}_{m,p}=\{\overline{z}\in p\ZZ/p^{m-1}\ZZ \mid \ord_{p}(z)\leq m-2\}$. Therefore we have
\[
\sum_{\substack{\overline{x}\in (p\ZZ/p^{m}\ZZ)^{n},\\ \ord_{p}(f(x))\leq m-2}}p^{-mn}\exp\Big(\frac{2\pi i f(x)}{p^{m}}\Big)=0. \qedhere
\]
\end{proof}

%
%
In the proof of the following lemma we will introduce again a constructible function $G$, similar to the one from the previous proof. For this exponential sum the different values $z$ of $f$ do not cancel out completely. By using the Lang-Weil estimation (see \cite{Lawe}) and Theorem \ref{02} we obtain the following upper bound for the second subsum.

\begin{lem}\label{16}
Let $f \in \ZZ[x_1, \ldots, x_n]$ be a non-constant polynomial, such that $f(0) = 0$. Put $\sigma = \min\{c_0(f), \frac{1}{2}\}$, where $c_0(f)$ is the log-canonical threshold of $f$ at $0$. Then there exist, for each integer $m>1$, a natural number $N_{m}$ and a positive constant $D_{m}$, such that, for all $p>N_{m}$, we have
\[
\Big\lvert\sum_{\substack{\overline{x}\in (p\ZZ/p^{m}\ZZ)^{n},\\ \ord_{p}(f(x))= m-1}}p^{-mn}\exp\left(\frac{2\pi if(x)}{p^{m}}\right)\Big\rvert \leq D_{m}p^{-m\sigma}.
\]
\end{lem}

\begin{proof}
Let $\phi,\overline{\phi}$ be two $\cL_{\ZZ}$-formulas given by 
\begin{align*}
&\phi(x_{1},\ldots,x_{n},z,m)=\\
&\bigwedge_{i=1}^{n}(\ord(x_{i})\geq 1)\wedge(\ord(z)=m-1)\wedge(\ord(z-f(x_{1},\ldots,x_{n}))\geq m),\\
&\overline{\phi}(x_{1},\ldots,x_{n},\xi,m)=\\
&\bigwedge_{i=1}^{n}(\ord(x_{i})\geq 1) \wedge (\ord(f(x_{1},\ldots,x_{n})=m-1)\wedge(\ac(f(x_{1},\ldots,x_{n}))=\xi),
\end{align*}
where $x_{i},z$ are in the valued field-sort, $m$ is in the valued group-sort and $\xi$ is in the residue field-sort. To shorten notation we set $x=(x_{1},\ldots,x_{n})$. For each prime $p$, we fix a uniformiser $\varpi_{p}$ of $\QQ_{p}$, then $\phi,\overline{\phi}$ define, for each $p$, two definable sets $X_{p}\subset(p\ZZ_{p})^{n}\times\ZZ_{p}\times\ZZ$  and $\overline{X}_{p}\subset(p\ZZ_{p})^{n}\times\FF_{p}\times\ZZ$ given by
\[
X_{p}=\{(x,z,m)\in (p\ZZ_{p})^{n}\times\ZZ_{p}\times\ZZ \mid \ord_p(f(x)-z)\geq m, \ord_p(z)=m-1\}
\]
and
\[
\overline{X}_{p}=\{(x,\xi,m)\in (p\ZZ_{p})^{n}\times\FF_{p}\times\ZZ \mid \ord_p(f(x))=m-1, \ac_{\varpi_{p}}(f(x))=\xi\}.
\]
It is obvious that $X_{p}$ does not depend on $\varpi_{p}$.

We denote by $X\subset h[n+1,0,1]$, resp.\  $\overline{X} \subset h[n,1,1]$, the definable subassignments defined by $\phi$, resp.\ $\overline{\phi}$. Let $F:=\textbf{1}_{X}\in I_{h[0,0,1]}(h[n+1,0,1])$ and $\pi$ the projection from $h[n+1,0,1]$ to $h[1,0,1]$. Then we have $G:=\pi_{!}(F)\in I_{h[0,0,1]}(h[1,0,1])$. For each prime $p$ and each uniformiser $\varpi_{p}$ of $\QQ_{p}$, there exist the following interpretations of $F$ and $G$ in $\QQ_{p}$.
\[
F_{\varpi_{p}}=\textbf{1}_{X_{p}}
\]
and 
\[
G_{\varpi_{p}}(z,m)=\int_{X_{p,z,m}}|dx|=p^{-mn}\#\{\overline{x}\in (p\ZZ/p^{m}\ZZ)^{n} \mid f(x)\equiv z\text{ mod }p^{m}\},
\]
if $\ord_p(z)= m-1$, where $X_{p,z,m}$ is the fiber of $X_{p}$ over $(z,m)$, and
\[
G_{\varpi_{p}}(z,m)=0,
\]
if $\ord_p(z)\neq m-1$. We can see that both $F_{\varpi_{p}}(x,z,m)$ and $G_{\varpi_{p}}(z,m) $ do not depend on $\varpi_{p}$. So we can set  $G(z,m,p):=G_{\varpi_{p}}(z,m)$. The idea is to partition $p^{m-1}\ZZ_p \backslash p^m\ZZ_p$ into sets on which $G(\cdot,m,p)$ is constant. First of all, we can see that $G(\cdot,m,p)$ is constant on balls of the form
\[
\{z\in\ZZ_{p} \mid \ord_{p}(z)=m-1,\ac_{\varpi_{p}}(z)=\xi_{0}\},
\]
with $\xi_{0}\in\FF_{p}^{\times}$.
Now we will look more closely on which of these balls $G(\cdot,m,p)$ takes the same value. In what follows we will show is that for $p$ big enough, if $\varpi_p, \varpi'_p$ are two uniformiser, then $G(\cdot, m,p)$ will be the same on the sets $\{z\in\ZZ_{p} \mid \ord_p(z)=m-1,\ac_{\varpi_{p}}(z)=\xi_{0}\}$ and $\{z\in\ZZ_{p} \mid \ord_p(z)=m-1,\ac_{\varpi'_{p}}(z)=\xi_{0}\}$. When this holds, we can see that $G$ will be constant on the orbits of an action of the group $\mu_{p-1}(\QQ_{p})$ on $\QQ_{p}$.\\

We take $\overline{F}:=\textbf{1}_{\overline{X}}\in I_{h[0,0,1]}(h[n,1,1])$ and $\overline{\pi}$ the projection from $h[n,1,1]$ to $h[0,1,1]$. Then we have $\overline{G}:=\overline{\pi}_{!}(\overline{F})\in I_{h[0,0,1]}(h[0,1,1])$. For each prime $p$ and each uniformiser $\varpi_{p}$ of $\QQ_{p}$, there exist the following interpretations of $\overline{F}$ and $\overline{G}$ in $\QQ_{p}$.
\[
\overline{F}_{\varpi_{p}}=\textbf{1}_{\overline{X}_{p}}
\]
and 
\begin{align*}
\overline{G}_{\varpi_{p}}(\xi,m)&=\int_{\overline{X}_{p,\xi,m}}|dx|\\
&=p^{-mn}\#\{\overline{x}\in (p\ZZ/p^{m}\ZZ)^{n} \mid \ord_p(f(x))= m-1,\ac_{\varpi_{p}}(f(x))=\xi\},
\end{align*}
 where $\overline{X}_{p,\xi,m}$ the fiber of $\overline{X}_{p}$ over $(\xi,m)$.

Since $\overline{G}\in I_{h[0,0,1]}(h[0,1,1])$, we can write $\overline{G}$ in  the form 
\[
\overline{G}(\xi,m)=\sum_{i\in I}n_{i}\alpha_{i}(\xi,m)\LL^{\beta_{i}(\xi,m)}[V_{i}],
\]
where $n_{i}\in \ZZ$, $\alpha_{i},\beta_{i}$ are $\cL_{\ZZ}$-definable functions from  $h[0,1,1]$ to $h[0,0,1]$ and $[V_{i}]\in K_{0}(\text{RDef}_{h[0,1,1],\cL_{\ZZ}})$.
We use elimition of quantifiers (Corollary \ref{QE}) for the formulas defining $\alpha_{i},\beta_{i}, V_{i}$, hence there exist $N\in\NN$, and $(\cL_{\text{ring}} \cup \ZZ)$-formulas $\phi_{ij},\theta_{ij},\varsigma_{ij}$ and $(\cL_{\text{oag}} \cup \ZZ)$-formulas $\eta_{ij},\nu_{ij},\tau_{ij}$, where $j\in J$, such that for all $p>N$ and all uniformiser $\varpi_{p}$, we have
\begin{align*}
\alpha_{i,\varpi_p}(\xi,m)=\eta&\Leftrightarrow\vee_{j\in J}(\phi_{ij}(\xi)\wedge\eta_{ij}(\eta,m));\\
\beta_{i,\varpi_p}(\xi,m)=\nu&\Leftrightarrow\vee_{j\in J}(\theta_{ij}(\xi)\wedge\nu_{ij}(\nu,m));\\
(\xi,m,\varsigma)\in V_{i,\varpi_{p}}&\Leftrightarrow\vee_{j\in J}(\varsigma_{ij}(\varsigma,\xi)\wedge\tau_{ij}(m)).
\end{align*}
From these formulas we can see that $\overline{G}_{\varpi_{p}}(\xi,m)$ does not depend on the uniformiser $\varpi_{p}$, so we will write $\overline{G}(\xi,m,p)$ instead of $\overline{G}_{\varpi_{p}}(\xi,m)$. But by definition of $G$ and $\overline{G}$ we can see that $G(z,m,p) = G_{\varpi_{p}}(z,m)=\overline{G}_{\varpi_{p}}(\xi,m)=\overline{G}(\xi,m,p)$, if $\ac_{\varpi_{p}}(z)=\xi$ and $\ord{_p}(z)=m-1$. Therefore, for $m>1$, $p>N$ and $\ord{_p}(z_{1})=\ord{_p}(z_{2})=m-1$, we have $G(z_{1},m,p)=G(z_{2},m,p)$, if there exist two uniformisers $\varpi_{1,p},\varpi_{2,p}$ such that $\ac_{\varpi_{1,p}}(z_{1})=\ac_{\varpi_{2,p}}(z_{2})\in\FF_{p}^{\times}$. Let $d=\gcd(m-1,p-1)$, then by the same reasoning as in Lemma \ref{14} we have that $G(\cdot,m,p)$ will be constant on the sets
\[
\Big\{z \mid \ord{_p}(z)=\ord{_p}(z_{0})=m-1\wedge\ac_{p}\Big(\dfrac{z}{z_{0}}\Big)^{\frac{p-1}{d}}=1\Big\},
\]
for any $z_{0}\in \ZZ_p$ with $\ord{_p}(z_{0})=m-1$. So we can decompose $p^{m-1}\ZZ_p \backslash p^m\ZZ_p$ into $d$ of these sets, each of them will consist of $\frac{p-1}{d}$ disjoint balls of volume $p^{-m}$ and $G(\cdot,m,p)$ will be constant on these sets. We denote  these sets by $Y_{1},\ldots,Y_{d}$ and the values of $G(\cdot,m,p)$ on these sets by $G_{1},\ldots,G_{d}$ respectively. We remark that if $\ord_{p}(z)=m-1$, then 
\[
\exp\Big(\frac{2\pi iz}{p^{m}}\Big)=\exp\Big(\frac{2\pi i \ \ac_{p}(z)}{p}\Big),
\]
so 
\begin{align*}
\Big|\sum_{\overline{y}\in Y_{i}/p^{m}\ZZ_p}\exp\Big(\frac{2\pi iy}{p^{m}}\Big)\Big|
&=\Big|\sum_{\xi\in\ac_p(Y_{i})}\exp\Big(\frac{2\pi i\xi}{p}\Big)\Big|\\
&=\Big|\sum_{u\in\FF_{p}^{\times}}\exp\Big(\frac{2\pi iu^{d}\xi_{0}}{p}\Big)\Big|\\
\end{align*}
for any $\xi_{0}\in\ac(Y_{i})$. By the last result from \cite{Weil} we have 
\[
\Big|\sum_{u\in\FF_{p}^{\times}}\exp\Big(\frac{2\pi iu^{d}\xi_{0}}{p}\Big)\Big| = \Big|\sum_{u\in\FF_{p}}\exp\Big(\frac{2\pi iu^{d}\xi_{0}}{p}\Big) - 1 \Big| \leq (d-1)p^{\frac{1}{2}} + 1 \leq dp^{\frac{1}{2}},
\]
hence
\begin{align*}
&\Big|\sum_{\substack{\overline{x}\in (p\ZZ/p^{m}\ZZ)^{n},\\ \ord_{p}(f(\overline{x}))= m-1}}p^{-mn}\exp\Big(\frac{2\pi i f(x)}{p^{m}}\Big)\Big|=\\
&\Big|\sum_{0\neq \overline{z} \in p^{m-1}\ZZ_{p}/p^{m}\ZZ_{p}}G(z,m,p)\exp\Big(\frac{2\pi iz}{p^{m}}\Big)\Big|=\\
&\Big|\sum_{i=1}^{d}G_{i}\sum_{\overline{y}\in Y_{i}/p^m\ZZ_p}\exp\Big(\frac{2\pi iy}{p^{m}}\Big)\Big| \leq \Big|\sum_{i=1}^{d}G_{i}dp^{\frac{1}{2}}\Big|.
\end{align*}
We also have
\begin{align*}
\sum_{i=1}^{d}\dfrac{p-1}{d}G_{i}&=\sum_{0\neq \overline{z} \in p^{m-1}\ZZ_{p}/p^{m}\ZZ_{p}}G(z,m,p)\\
&=p^{-mn}\#\{\overline{x}\in (p\ZZ/p^{m}\ZZ)^{n} \mid \ord_{p}(f(x))=m-1\}\\
&=p^{-mn}\#A_{p,m}
\end{align*}
where $A_{p,m}:=\{\overline{x}\in (p\ZZ_{p}/p^{m}\ZZ_{p})^{n} \mid \ord_{p}(f(x))=m-1\}$. When we view $A_{p,m}$ as a subvariety of $\FF_{p}^{mn}$, then, by the Lang-Weil estimation (see \cite{Lawe}), there exists a constant $D'_{m}$, not depending on $p$, such that 
\[
\#A_{p,m}=D'_{m}p^{\dim_{\FF_{p}}(A_{p,m})}+O(p^{\dim_{\FF_{p}}(A_{p,m})-\frac{1}{2}}).
\]
By Theorem \ref{02} we have 
\[
c_{0}(f)\leq \dfrac{(m-1)n-\dim_{\FF_{p}}(\tilde{A}_{p,m})}{m-1},
\]
where $\tilde{A}_{p,m}$ is the image of $A_{p,m}$ under the projection $\pi_{m}:(\ZZ_{p}/p^{m}\ZZ_{p})^{n}\rightarrow (\ZZ_{p}/p^{m-1}\ZZ_{p})^{n}$, viewed as a subvariety of $\FF_{p}^{mn-n}$. Then we have
\[
\dim_{\FF_{p}}(A_{m,p})\leq n+ \dim_{\FF_{p}}(\tilde{A}_{m,p})\leq mn-(m-1)c_{0}(f).
\]
And now we finish the proof by showing that for all $p$ big enough,
\begin{align*}
\Big|\sum_{\substack{\overline{x}\in (p\ZZ/p^{m}\ZZ)^{n},\\\ord_{p}(f(x))= m-1}}p^{-mn}\exp\Big(\frac{2\pi if(x)}{p^{m}}\Big)\Big|&\leq \Big|\sum_{i=1}^{d}G_{i}dp^{\frac{1}{2}}\Big|\\
&=d^2\frac{p^{-mn+\frac{1}{2}}}{p-1}\#A_{p,m}\\
&\leq 2d^2p^{-mn-\frac{1}{2}} D'_m p^{mn-(m-1)c_{0}(f)}\\
&\leq D_mp^{-m\sigma},
\end{align*}
because $\sigma=\min\big\{\frac{1}{2},c_{0}(f)\big\}$. Here $D_m = 2(m-1)^2D'_{m}$.
\end{proof}

The last subsum can be easily estimated by use of the Lang-Weil estimation (see \cite{Lawe}) and Theorem \ref{02}.

\begin{lem}\label{17}
Let $f \in \ZZ[x_1, \ldots, x_n]$ be a non-constant polynomial, such that $f(0) = 0$. Put $\sigma = \min\{c_0(f), \frac{1}{2}\}$, where $c_0(f)$ is the log-canonical threshold of $f$ at $0$. Then there exist, for each integer $m>1$, a natural number $N_{m}$ and a positive constant $D_{m}$, such that, for all $p>N_{m}$, we have
\[
\Big|\sum_{\substack{\overline{x}\in (p\ZZ/p^{m}\ZZ)^{n},\\\ord_{p}(f(x))\geq m}}p^{-mn}\exp\Big(\frac{2\pi if(x)}{p^{m}}\Big)\Big|\leq D_{m}p^{-m\sigma}.
\]
\end{lem}

\begin{proof}
If $\ord_{p}(f(x))\geq m$, then $\exp\Big(\frac{2\pi if(x)}{p^{m}}\Big)=1$ so we have 
\[
\Big|\sum_{\substack{\overline{x}\in (p\ZZ/p^{m}\ZZ)^{n},\\\ord_{p}(f(x))\geq m}}p^{-mn}\exp\Big(\frac{2\pi if(x)}{p^{m}}\Big)\Big|=p^{-mn}\#B_{p,m},
\]
where $B_{p,m}:=\{\overline{x}\in (p\ZZ_{p}/p^{m}\ZZ_{p})^{n} \mid \ord_{p}(f(x))\geq m\}$. We can view $B_{p,m}$ as a subvariety of $\FF_{p}^{mn}$. Then by the Lang-Weil estimation (see \cite{Lawe}), there exists a number $D_{m}$, which does not depend on $p$, such that 
\[
\#B_{p,m}=D_{m}p^{\dim_{\FF_{p}}(B_{m,p})}+O(p^{\dim_{\FF_{p}}(B_{m,p})-\frac{1}{2}}).
\]
By Theorem \ref{02} we have $c_{0}(f)\leq\dfrac{mn-\dim_{\FF_{p}}(B_{m,p})}{m}$, so $\dim_{\FF_{p}}(B_{m,p})\leq mn-mc_{0}(f)$.
Hence, for all $p$ big enough,
\begin{align*}
\Big|\sum_{\substack{\overline{x}\in (p\ZZ/p^{m}\ZZ)^{n},\\\ord_{p}(f(x))\geq m}}p^{-mn}\exp\Big(\frac{2\pi if(x)}{p^{m}}\Big)\Big|&\leq p^{-mn}D_{m}p^{mn-mc_{0}(f)}\\
&\leq D_{m}p^{-m\sigma}. \qedhere
\end{align*}
\end{proof}

We will now put the three lemmas together to prove one of our main theorems. The essential ingredient in this proof is the expression that was obtained in Corollary \ref{07}.

\begin{proof}[Proof of the Main Theorem \ref{*}]
From the Lemmas \ref{14}, \ref{16} and \ref{17} it follows that, for each $m>1$, there exists a natural number $N_{m}$ and a positive constant $C_{m}$, such that for all $p>N_{m}$, we have 
\begin{equation}\label{eq: lemmas combined}
|E_{m,p}^{0}(f)|\leq C_{m}p^{-\sigma m}.
\end{equation}
By Corollary \ref{07} (with $\text{Supp}(\overline{\Phi})=\{0\}$), there exist constants $s, M',N' \in\NN$, and for each $1\leq i\leq s$, there exist constants $\beta_{i}\in \NN$, $\lambda_{i}\in\QQ$  and a definable set $A_{i}\subset\NN$ in the Presburger language $\mathcal{L}_{\text{Pres}}$, such that for all $p>N'$ and for all $1\leq i\leq s$, there exists $a_{i,p}\in\CC$ for which the formula
\[
E_{m,p}^{0}(f)=\sum_{i=1}^{s}a_{i,p}m^{\beta_{i}}p^{-\lambda_{i}m}\11_{A_{i}}(m)
\]
holds, for all $m>M'$. Moreover from the results in Section \ref{***} we can deduce that $0\leq \beta_{i}\leq n-1$ and $c_{0}(f)\leq \lambda_{i}$ for all $1 \leq i \leq s$.
After enlarging $M'$ and removing some small elements from $A_{i}$, we can assume that, for each subset $I \subset \{1, \ldots, s\}$, the set $\cap_{i\in I} A_{i}\backslash\cup_{i\notin I} A_{i}$ is either empty or infinite. Notice that for each $m >M'$, there is a unique subset $I \subset \{1, \ldots, s\}$, such that $m \in \cap_{i\in I} A_{i}\backslash\cup_{i\notin I} A_{i}$. 

\begin{claim}\label{claim}
There exist $M_0>M', N_0>N'$ and a positive  constant $C_{0}$, such that for all $m>M_{0}, p>N_{0}$ and $1\leq i\leq s$, we have
\[
|a_{i,p}p^{-\lambda_{i}m}|\leq C_{0}p^{-\sigma m}.
\]
\end{claim}

Since there are only finitely many subsets $I \subset \{1, \ldots, s\}$, it is sufficient to fix a subset $I$ and prove the claim for $m$ restricted to the set $\cap_{i\in I} A_{i}\backslash\cup_{i\notin I} A_{i}$. Without loss of generality, we can assume that $I=\{1,\ldots,r\}$. If $p>N'$, $m\in \cap_{i\in I} A_{i}\backslash\cup_{i\notin I} A_{i}$ and $m>M'$, then we have
\[
E_{m,p}^{0}(f)=\sum_{i=1}^{r}a_{i,p}m^{\beta_{i}}p^{-\lambda_{i}m}.
\]
From Equation \ref{eq: lemmas combined} we can see that, for such $m$ and for all $p>\max\{N',N_{m}\}$, we have 
\[
|E_{m,p}^{0}|=\Big|\sum_{i=1}^{r}a_{i,p}m^{\beta_{i}}p^{-\lambda_{i}m}\Big|\leq C_{m}p^{-\sigma m}.
\]
This implies that 
\[
\Big|\sum_{i=1}^{r}a_{i,p}m^{\beta_{i}}p^{(\sigma-\lambda_{i})m}\Big|\leq C_{m}.
\]
It is easy to see that there exist $m_{1},\ldots,m_{r}\in\cap_{i\in I} A_{i}\backslash\cup_{i\notin I} A_{i}$, all bigger than $M'$, and $N_{I}>\max\{N', N_{m_1}, \ldots, N_{m_r}\}$, such that all of the determinants of the size $r$ and $r-1$ submatrices of the matrix $B_{p}=(m_{j}^{\beta_{i}}p^{(\sigma-\lambda_{i})m_{j}})_{1\leq j,i\leq r}$ are different from zero for every $p>N_{I}$. We set
\begin{align*}
C_{I}&:=\max\{C_{m_{i}} \mid 1\leq i\leq r\};\\
c_{j,p}&:=\sum_{i=1}^{r}a_{i,p}m_{j}^{\beta_{i}}p^{(\sigma-\lambda_{i})m_{j}}, \text{ for }1\leq j\leq r;\\
D_{p}&:=\det(B_{p});\\
D_{k,l,p}&:=(-1)^{k+l}\det\big((m_{j}^{\beta_{i}}p^{(\sigma-\lambda_{i})m_{j}})_{j\neq k, i\neq l}\big),\text{ for }1\leq k,l \leq r.
\end{align*}
If we write $x_{p}=(a_{1,p},\ldots,a_{r,p})^{T}$ and $c_{p}=(c_{1,p},\ldots,c_{r,p})^{T}$, then $x_{p}$  is a solution of the equation $B_{p}x=c_{p}$. By our assumption on $m_{1},\ldots,m_{r}$ we see that $D_{p}\neq 0$ and $D_{k,l,p}\neq 0$ for every $1\leq k,l\leq r$ and for $p>N_{I}$. Using Cramer's rule we have 
\[
a_{i,p}=\dfrac{\sum_{j=1}^{r}c_{j,p}D_{j,i,p}}{D_{p}},
\]
for all $1\leq i\leq r$ and $p>N_{I}$. We remark that $|c_{j,p}|\leq C_{I}$, for all $p>N_{I}$, and that $\lambda_{i}\geq \sigma$, for all $1\leq i\leq r$.
This gives us
\[
|a_{i,p}|\leq \dfrac{\sum_{j=1}^{r}|c_{j,p}D_{j,i,p}|}{|D_{p}|}\leq C_{I}\dfrac{\sum_{j=1}^{r}|D_{j,i,p}|}{|D_{p}|},
\]
for all $1 \leq i \leq r$ and $p>N_I$. Then, by the definition of determinant, there exists $\alpha$, such that, for $1\leq i\leq r$ and $p>N_{I}$ we have $|a_{i,p}|\leq p^{\alpha}$. Let $1\leq i_{0}\leq r$, we will now distinguish two cases.\\
If $\lambda_{i_{0}}>\sigma$, then there exists $M_{i_{0}}>M'$ such that, for every $m>M_{i_{0}}$ and $p>N_I$ we have 
\[
|a_{i_{0},p}p^{-m\lambda_{i_{0}}}|\leq p^{\alpha-m\lambda_{i_{0}}}\leq p^{-m\sigma}.
\]
If $\lambda_{i_{0}}=\sigma$, we observe that 
\[
D_{p}=\sum_{j=1}^{r}m_{j}^{\beta_{i_{0}}}p^{(\sigma-\lambda_{i_{0}
})m_{j}}D_{j,i_{0},p} = \sum_{j=1}^{r}m_{j}^{\beta_{i_{0}}}D_{j,i_{0},p}.
\]
By the definition of determinant, there exist $\gamma_{j}, d_{j}$, for each $1\leq j \leq r$, such that $D_{j, i_{0},p}=d_{j}p^{\gamma_{j}}$, when $p\rightarrow\infty$. By changing $m_{1},\ldots,m_{r}$ if necessary, we can assume that there exists $d>0$, such that $|D_{p}|=dp^{\gamma}$, when $p\rightarrow\infty$, where $\gamma=\max\{\gamma_{j} \mid 1\leq j\leq r\}$. Thus there exist $C_{0}>0$ and $N_{i_{0}} >N_I$, such that 
\[
|a_{i_{0},p}|\leq C_{I}\dfrac{\sum_{j=1}^{r}|D_{j,i_0,p}|}{|D_{p}|}\leq C_{0},
\]
for all $p>N_{i_{0}}$. And so
\[
|a_{i_{0},p}p^{-m\lambda_{i_{0}}}|\leq C_{0}p^{-m\sigma},
\]
for all $p>N_{i_{0}}$ and all $m>1$.
This proves the claim.\\

Hence we have 
\[
|E_{p,m}^{0}(f)|=\Big|\sum_{i=1}^{s}a_{i,p}m^{\beta_{i}}p^{-\lambda_{i}m}\11_{A_{i}}(m)\Big|\leq sC_{0}m^{n-1}p^{-m\sigma},
\]
for all $m>M_{0}, p>N_{0}$. By Equation \ref{eq: lemmas combined} we also have, for each $1<m \leq M_0$, an upper bound for $|E_{m,p}^0(f)|$ in terms of some constant $C_m$.
Now let $N:=\max\{N_{i} \mid i\in\{0,2,\ldots,M_{0}\}\}$ and $C:=\max\{sC_{0},C_{2},\ldots,C_{M_{0}}\}$, then we have 
\[
|E_{p,m}^{0}|=\Big|\sum_{i=1}^{s}a_{i,p}m^{\beta_{i}}p^{-\lambda_{i}m}\11_{A_{i}}(m)\Big|\leq Cm^{n-1}p^{-m\sigma},
\]
for all $m>1, p>N$.
\end{proof}

\section{The second approach by geometry}\label{sec: geometry}
We take $f \in \ZZ[x_1, \ldots, x_n]$ a nonconstant polynomial with $f(0)=0$ and we put $\sigma = \min\{c_0(f), \frac{1}{2}\}$, where $c_0(f)$ is the log-canonical threshold of $f$ at $0$. We use the notation of Section \ref{03} with $(Y,h)$ an embedded resolution of $f^{-1}(0)$, $K=\QQ$ and $\cO_{K}=\ZZ$. Then by Theorem \ref{06} and the discussion preceding that theorem, there exist $M_{0}, N_{0} \in \NN$, such that for all $p>N_0$, there exist at most $M_0$ non-trivial characters $\chi$ of $\ZZ_{p}^{\times}$ with $Z_{\Phi_{p}}(p,\chi,s,f)\neq 0$, where $\Phi_{p}=\11_{{(p\ZZ_p)^n}}$, i.e., $\text{Supp}(\overline{\Phi}_p) = \{0\}$. Moreover any such character has conductor $c(\chi)=1$. To simplify we will omit $\Phi_{p}$ and $f$ in the notation of Igusa's local zeta functions.

We can suppose that $f$ has good reduction mod $p$  for all $p>N{_0}$ (after enlarging $N{_0}$ if necessary). Let $p >N_0$ and let $E$ be an irreducible component of $h^{-1}(Z(f))$, such that $0\in \overline{h}(\overline{E})$, then $h(E)\cap p\ZZ_{p}^{n}\neq\emptyset$. Remark that $h$ is proper, so $h(E)$ is a closed subvariety of $\AA^{n}$. Therefore, after possibly enlarging $N_0$ again, we can assume that if $0\notin h(E)$, then $h(E)\cap p\ZZ_{p}^{n}=\emptyset$, for all $p > N_0$. Hence, for $p > N_0$, $0 \in \overline{h}(\overline{E})$ implies $0\in h(E)$. 
So the map $E \mapsto \overline{E}$ is a bijection between
\[
\{E_i \mid i \in T,\ 0\in h(E_i)\} \qquad \text{and} \qquad \{\overline{E}_i \mid i \in T,\ 0\in \overline{h}(\overline{E}_i)\},
\]
where $T$ is as in Section \ref{***},
hence
\begin{equation}\label{eq: lct reduced mod p}
c_{0}(f)=\min_{i \in T : 0\in\overline{h}(\overline{E}_i(\FF_{p}))}\big\{\frac{\nu_{i}}{N_{i}}\big\}.
\end{equation}
Now to prove the Main Theorem \ref{*}, we use Proposition \ref{04} for $p>N_0$, $u=1$, $\pi=p$ and $m>1$. This tells us that $E_{p,m}^{0}(f)$ is equal to 
\[
Z(p,\chi_{\text{triv}},0)+\text{Coeff}_{t^{m-1}}\dfrac{(t-p)Z(p,\chi_{\text{triv}},s)}{(p-1)(1-t)}+\sum_{\chi\neq\chi_{\text{triv}}}g_{\chi^{-1}}
\text{Coeff}_{t^{m-1}}Z(p,\chi,s).
\]

\begin{lem}\label{18}There exist a positive constant $C$ and a natural number $N$, such that for all $m>1$, $p>N$, we have 
\[
Z(p,\chi_{\textnormal{triv}},0)+\textnormal{Coeff}_{t^{m-1}}\dfrac{(t-p)Z(p,\chi_{\textnormal{triv}},s)}{(p-1)(1-t)}\leq Cm^{n-1}p^{-mc_0(f)}.
\]
\end{lem}

\begin{proof}
We use Theorem \ref{05} which tells us that there exists a natural number $N'$, such that for all $p>N'$,
\begin{align}
\label{eq: zeta zero}Z(p,\chi_{\text{triv}},0)&=p^{-n}\sum_{I\subset T}c_{I,\chi_{\text{triv}}}^{0}\prod_{i\in I}\frac{(p-1)p^{-\nu_{i}}}{1-p^{-\nu_{i}}};\\
Z(p,\chi_{\text{triv}},s)&=p^{-n}\sum_{I\subset T}c_{I,\chi_{\text{triv}}}^{0}\prod_{i\in I}\frac{(p-1)t^{N_{i}}p^{-\nu_{i}}}{1-t^{N_{i}}p^{-\nu_{i}}}. \nonumber
\end{align}
From the formula
$\frac{(t-p)}{(p-1)(1-t)}=-\frac{1}{p-1}-\frac{1}{1-t}$ we get
\[
\text{Coeff}_{t^{m-1}}\frac{(t-p)Z(p,\chi_{\text{triv}},s)}{(p-1)(1-t)}=-\text{Coeff}_{t^{m-1}}\frac{Z(p,\chi_{\text{triv}},s)}{p-1}-\text{Coeff}_{t^{m-1}}\frac{Z(p,\chi_{\text{triv}},s)}{1-t},
\]
where
\begin{align}
\label{eq: coeff zeta over p-1}\text{Coeff}_{t^{m-1}}\frac{Z(p,\chi_{\text{triv}},s)}{p-1}&=\sum_{I\subset T}p^{-n}c_{I,\chi_{\text{triv}}}^{0}(p-1)^{\#I}\frac{1}{p-1} \text{Coeff}_{t^{m-1}}\prod_{i\in I}\frac{t^{N_{i}}p^{-\nu_{i}}}{1-t^{N_{i}}p^{-\nu_{i}}};\\
\label{eq: coeff zeta over 1-t}\text{Coeff}_{t^{m-1}}\frac{Z(p,\chi_{\text{triv}},s)}{1-t}&=\sum_{I\subset T}p^{-n}c_{I,\chi_{\text{triv}}}^{0}(p-1)^{\#I} \text{Coeff}_{t^{m-1}}\frac{1}{1-t}\prod_{i\in I}\frac{t^{N_{i}}p^{-\nu_{i}}}{1-t^{N_{i}}p^{-\nu_{i}}}.
\end{align}
Notice that if $I \subset T$, such that $\overset{\circ}{\overline{E}}_{I}\cap \overline{h}^{-1}(0) = \emptyset$, then $c_{I,\chi_{\text{triv}}}^0 = 0$. Hence we can assume that $\overset{\circ}{\overline{E}}_{I}\cap \overline{h}^{-1}(0)\neq\emptyset$. For such $I \subset T$ we have
\begin{align}
\label{eq: coeff}\text{Coeff}_{t^{m-1}}\prod_{i\in I}\frac{t^{N_{i}}p^{-\nu_{i}}}{1-t^{N_{i}}p^{-\nu_{i}}}&=\sum_{(a_{i})_{i\in I}\in J_{I,m}}p^{-\sum_{i\in I}\nu_{i}(a_{i}+1)}; \\
\label{eq: coeff2}\text{Coeff}_{t^{m-1}}\frac{1}{1-t}\prod_{i\in I}\frac{t^{N_{i}}p^{-\nu_{i}}}{1-t^{N_{i}}p^{-\nu_{i}}}&=\sum_{(a_{i})_{i\in I}\in J'_{I,m}}p^{-\sum_{i\in I}\nu_{i}(a_{i}+1)},
\end{align}
where $J_{I,m}:=\{(a_{i})_{i\in I} \in \NN^{\#I} \mid \sum_{i\in I}N_{i}(a_{i}+1)=m-1\}$ and $J'_{I,m}:=\{(a_{i})_{i\in I} \in \NN^{\#I} \mid \sum_{i\in I}N_{i}(a_{i}+1)\leq m-1\}$.
When $(a_{i})_{i\in I}\in J_{I,m}$ and $p>N_0$, we can use Equation \ref{eq: lct reduced mod p} for the following estimate:
\begin{align*}
-\sum_{i\in I}\nu_{i}(a_{i}+1)&=-\sum_{i\in I}N_{i}\sigma_{i}(a_{i}+1)\\
&=-\sum_{i\in I}N_{i}(a_{i}+1)(\sigma_{i}-c_0(f))-(m-1)c_0(f)\\
&\leq -(m-1)c_0(f),
\end{align*}
where $\sigma_{i}=\frac{\nu_{i}}{N_{i}}\geq c_{0}(f)$, since we assumed that $\overset{\circ}{\overline{E}}_{I}\cap \overline{h}^{-1}(0)\neq\emptyset$. We also deduce from this assumption that $\#I \leq n$, thus by Equation \ref{eq: coeff},
\begin{equation} \label{eq: sum1}
\text{Coeff}_{t^{m-1}}\prod_{i\in I}\frac{t^{N_{i}}p^{-\nu_{i}}}{1-t^{N_{i}}p^{-\nu_{i}}}\leq \#(J_{I,m})p^{-(m-1)c_0(f)}\leq m^{n-1}p^{-(m-1)c_0(f)},
\end{equation}
for all $p > N_0$. Using Equation \ref{eq: coeff2} we can see that, in order to find an upper bound for the difference of \ref{eq: zeta zero} and \ref{eq: coeff zeta over 1-t}, we need to analyse the expression
\begin{align*}
&\prod_{i\in I}\frac{p^{-\nu_{i}}}{1-p^{-\nu_{i}}}-\sum_{(a_{i})_{i\in I}\in J'_{I,m}}p^{-\sum_{i\in I}\nu_{i}(a_{i}+1)}=\\
&\sum_{(a_{i})_{i\in I}\in\NN^{\#I}}p^{-\sum_{i\in I}\nu_{i}(a_{i}+1)}-\sum_{(a_{i})_{i\in I}\in J'_{I,m}}p^{-\sum_{i\in I}\nu_{i}(a_{i}+1)}=\\
&\sum_{(a_{i})_{i\in I}\in J''_{I,m}}p^{-\sum_{i\in I}\nu_{i}(a_{i}+1)},
\end{align*}
where $J''_{I,m}:=\{(a_{i})_{i\in I}\in \NN^{\#I} \mid \sum_{i\in I}N_{i}(a_{i}+1)\geq m\}$. Let $m_{I}:=m+\max\{N_{i} \mid i\in I\}$ and $\overline{J}_{I,m}:=\{(a_{i})_{i\in I} \in \NN^{\#I} \mid m\leq\sum_{i\in I}N_{i}(a_{i}+1)\leq m_{I}\}$. Afters some calculations we find that
\[
\sum_{(a_{i})_{i\in I}\in J''_{I,m}}p^{-\sum_{i\in I}\nu_{i}(a_{i}+1)}\leq \Big(1+\prod_{i\in I}\dfrac{1}{1-p^{-\nu_{i}}}\Big)\sum_{(a_{i})_{i\in I}\in \overline{J}_{I,m}}p^{-\sum_{i\in I}\nu_{i}(a_{i}+1)}.
\]
But if $(a_{i})_{i\in I}\in \overline{J}_{I,m+1}$, then, for all $p > N_0$,
\begin{align*}
-\sum_{i\in I}\nu_{i}(a_{i}+1)&=-\sum_{i\in I}N_{i}\sigma_{i}(a_{i}+1)\\
&\leq-\sum_{i\in I}N_{i}(a_{i}+1)(\sigma_{i}-c_0(f))-mc_0(f)\\
&\leq -mc_0(f).
\end{align*}
Therefore, for all $p>N_0$, we have,
 \begin{equation} \label{eq: sum2}
 \sum_{(a_{i})_{i\in I}\in J''_{I,m}}p^{-\sum_{i\in I}\nu_{i}(a_{i}+1)}\leq (1+2^{\#(I)})\#(\overline{J}_{I,m})p^{-mc_0(f)}\leq C_{I}m^{n-1}p^{-mc_0(f)},
 \end{equation}
where $C_{I}$ is a constant which does not depend on $m$ and $p$, for example $C_{I}=(1+2^{\#(I)}) (\max\{N_{i} \mid i\in I\}+1)$.\\

Now if $I\subset T$, then, by the
Lang-Weil estimate (see \cite{Lawe}), there exists a constant $D_{I}$ and a natural number $N_I$, depending only on $I$, such that for all $p >N_I$, we have
\begin{equation}\label{eq: LaWe estimation}
c_{I,\chi_{\text{triv}}}^0=\sum_{a\in \overset{\circ}{\overline{E}}_{I}\cap\overline{h}^{-1}(0)}\Omega_{\chi_{\text{triv}}}(a)=\#\big(\overset{\circ}{\overline{E}}_{I}\cap\overline{h}^{-1}(0)\big)\leq\#\big(\overset{\circ}{\overline{E}}_{I}\big) \leq D_{I}p^{n-\#I}.
\end{equation}
Putting together the inequalities \ref{eq: sum1}, \ref{eq: sum2} and \ref{eq: LaWe estimation} with the formulas \ref{eq: zeta zero}, \ref{eq: coeff zeta over p-1} and \ref{eq: coeff zeta over 1-t}, we find that there exists a natural number $N > \max\{N_0, N', (N_I)_{I\subset T}\}$, such that for all $p>N$, we have
\begin{align*}
&Z(p,\chi_{\text{triv}},0)+\text{Coeff}_{t^{m-1}}\frac{(t-p)Z(p,\chi_{\text{triv}},s)}{(p-1)(1-t)}\\
&\leq\sum_{{I\subset T: \overset{\circ}{\overline{E}}_{I}\cap\overline{h}^{-1}(0)\neq\emptyset}}p^{-n}c_{I,\chi_{\text{triv}}}^{0}(p-1)^{\#I}m^{n-1}\Big(\frac{p^{-mc_0(f)+c_0(f)}}{p-1}+C_{I}p^{-mc_0(f)}\Big)\\
&\leq\sum_{{I\subset T: \overset{\circ}{\overline{E}}_{I}\cap\overline{h}^{-1}(0)\neq\emptyset}}p^{-n}D_{I}p^{n-\#I}(p-1)^{\#I}m^{n-1}\Big(\frac{p^{-mc_0(f)+c_0(f)}}{p-1}+C_{I}p^{-mc_0(f)}\Big)\\
&\leq \sum_{{I\subset T: \overset{\circ}{\overline{E}}_{I}\cap\overline{h}^{-1}(0)\neq\emptyset}}D_{I}(C_{I}+2)m^{n-1}p^{-mc_0(f)} \leq Cm^{n-1}p^{-mc_0(f)},
\end{align*}
where $C=\sum_{I \subset T}D_{I}(C_{I}+2)$ is a constant that is independent of $p$ and $m$ and where we have used the fact that $c_0(f) \leq 1$.
\end{proof}

\begin{lem}\label{19}There exist a positive constant $C$ and a natural number $N$, such that for all $m>1$, $p>N$, we have 
\[
\Big|\sum_{\chi\neq\chi_{\textnormal{triv}}}g_{\chi^{-1}} \textnormal{Coeff}_{t^{m-1}}Z(p,\chi,s)\Big|\leq Cm^{n-1}p^{-m\sigma}.
\]
\end{lem}

\begin{proof}
We continue to use Theorem \ref{05}, hence there exists a natural number $N'$, such that for all $p>N'$,
\[
Z(p,\chi,s)=p^{-n}\sum_{{\substack{I\subset T,\\ \forall i\in I: d\mid N_{i}}}}c_{I,\chi}^{0}\prod_{i\in I}\frac{(p-1)t^{N_{i}}p^{-\nu_{i}}}{1-t^{N_{i}}p^{-\nu_{i}}},
\]
with $\chi$ a character of order $d$ on $\ZZ_{p}^{\times}$ with conductor $c(\chi)=1$.

For a subset $I\subset T$, such that $d|N_{i}, \forall i\in I$, and $\overset{\circ }{\overline{E}}_{I}\cap\overline{h}^{-1}(0)\neq\emptyset$, we have
\[
\text{Coeff}_{t^{m-1}}\prod_{i\in I}\frac{t^{N_{i}}p^{-\nu_{i}}}{1-t^{N_{i}}p^{-\nu_{i}}}=\sum_{(a_{i})_{i\in I}\in J_{I,m}}p^{-\sum_{i\in I}\nu_{i}(a_{i}+1)},
\]
where $J_{I,m}$ is as in the proof of Lemma \ref{18}. By the equations \ref{eq: coeff} and \ref{eq: sum1} we have
\[
\sum_{(a_{i})_{i\in I}\in J_{I,m}}p^{-\sum_{i\in I}\nu_{i}(a_{i}+1)}\leq m^{n-1}p^{-mc_0(f)+c_0(f)}.
\]
We use the Lang-Weil estimate (\cite{Lawe}) again, as we did in Lemma \ref{18}. So there exist a constant $D_I$ and a natural number $N_I$, depending only on $I$, such that  for all $p>N_I$, we have
\begin{align*}
|c_{I,\chi}^{0}|&=\Big|\sum_{a\in \overset{\circ}{\overline{E}}_{I}\cap\overline{h}^{-1}(0)}\Omega_{\chi}(a)\Big| \leq \sum_{a\in \overset{\circ}{\overline{E}}_{I}\cap\overline{h}^{-1}(0)}1 = \#\big(\overset{\circ}{\overline{E}}_{I}\cap\overline{h}^{-1}(0)\big)\\
&\leq\#\big(\overset{\circ}{\overline{E}}_{I}\big) \leq D_{I}p^{n-\#(I)}.
\end{align*}
If we take $N'' > \underset{I \subset T}{\max} N_I$, then we find that for all $p>N''$,
\begin{align*}
|\text{Coeff}_{t^{m-1}}Z(p,\chi,s)|&\leq\sum_{{\substack{I\subset T,\\ \forall i\in I: d\mid N_{i}}}}p^{-n}D_{I}p^{n-\#(I)}(p-1)^{\#(I)}m^{n-1}p^{-mc_0(f)+c_0(f)}\\
&\leq \sum_{{\substack{I\subset T,\\ \forall i\in I: d\mid N_{i}}}}D_{I}m^{n-1}p^{-mc_0(f)+c_0(f)}\\
&\leq C'm^{n-1}p^{-mc_0(f)+c_0(f)},
\end{align*}
where $C':=\sum_{{\substack{I\subset T}}}D_{I}$. Furthermore, by a standard result on Gauss sums, we can see that, if $\chi\neq\chi_{\text{triv}}$, then $|g_{\chi^{-1}}|\leq Dp^{-\frac{1}{2}}$, for some constant $D$, that does not depend on $\chi$ and $p$.
By Theorem \ref{06} and the discussiong preceding this theorem, we know that for $p>N_0$, the set $\Upsilon_{p}$ of non-trivial characters $\chi$ such that  $Z(p,\chi,s)$ is not zero, has atmost $M_{0}$ elements, for some positive integer $M_{0}$. Moreover, all these characters have conductor $1$. So there exists a natural number $N > \max\{N_0, N''\}$, such that for all $p>N$ and $m>1$, we have 
\begin{align*}
\Big|\sum_{\chi\neq\chi_{\text{triv}}}g_{\chi^{-1}} \text{Coeff}_{t^{m-1}}Z(p,\chi,s)\Big|
&=\Big|\sum_{\chi\in\Upsilon_{p}}g_{\chi^{-1}} \text{Coeff}_{t^{m-1}}Z(p,\chi,s)\Big|\\
&\leq\sum_{\chi\in\Upsilon_{p}}|g_{\chi^{-1}}|C'm^{n-1}p^{-mc_0(f)+c_0(f)}\\
&\leq \sum_{\chi\in\Upsilon_{p}}C'Dm^{n-1}p^{-m\sigma+\sigma-\frac{1}{2}}\\
&\leq Cm^{n-1}p^{-m\sigma}
\end{align*}
where $C=M_0 C' D$ is a constant that is independent of $p$ and $m$ and where we have used the fact that $\sigma = \min\{c_0(f), \frac{1}{2}\}$.
\end{proof}

\begin{remark}\label{rem: adapt proof}
These two proofs still work if we take $\Phi_p = \11_{U_p}$ instead of $\11_{(p\ZZ_p)^n}$, where $U_p$ is a union of some multiballs $y+(p\ZZ_p)^n$ in $\ZZ_p^n$, such that $C_{\overline{f}} \cap \overline{U}_p \subset \overline{f}^{-1}(0)$ (this is needed in the proof of Lemma \ref{19} to apply Theorem \ref{06}). We have to replace $c_0(f)$ for example by $c(f)$, $\overset{\circ}{\overline{E}}_{I}\cap\overline{h}^{-1}(0)$ by $\overset{\circ}{\overline{E}}_{I}\cap\overline{h}^{-1}(\overline{U}_p)$ and $c^0_{I, \chi_{\text{triv}}}$ by $c_{I, \11_{U_p}, \chi_{\text{triv}}}$. The constant $C$ and the natural number $N$ that are found in these proofs, do not depend on $U_p$. They do depend however on $f$ and on the embedded resolution $(Y,h)$ of $f$.
\end{remark}

\begin{proof}[Proof of the Main Theorem \ref{*}]
The proof follows by combining the two Lemmas \ref{18} and \ref{19} and using the fact that $\sigma \leq c_0(f)$.
\end{proof}

\section{Proof of the Main Theorem \ref{**} }\label{sec: Main Theorem 1.2}
In this section we will prove the Main Theorem \ref{**} by adapting the proofs from Section \ref{sec: geometry}. First, we need the following lemma.

\begin{lem}\label{20}
Let $f\in\ZZ[x_{1},\ldots,x_{n}]$ and $V_{f,p}$ be the set of critical values $z$ of $f$ in $\QQ_{p}$. Then $\#(V_{f,p})$ has an upper bound $d$, that does not depend on $p$. Furthermore, there exists $N$, such that for all $p>N$, the following holds:
\begin{enumerate}
\item for all $z \in V_{f,p}$, we have $\ord_p(z)= 0$;
\item for any two distinct points $z_{1}, z_{2}$ in $V_{f,p}$, we have $\ord_{p}(z_{1}-z_{2})=0$;
\item if $x \in \ZZ_p^n$ such that $\ord_p(f(x)-z)= 0$ for all $z \in V_{f,p}$, then $x$, resp.\ $\overline{x}$, is a regular point of $f$, resp.\ $\overline{f} := (f \mod p)$.
\end{enumerate}

\end{lem}

\begin{proof}Remark that we can uniquely extend the valuation $\ord_{p}$ to $\overline{\QQ}_{p}$ (the algebraic closure of $\QQ_p$). We denote by $\cO_{p}=\{z\in \overline{\QQ}_{p}|\ord_{p}(z)\geq 0\}$ the ring of integers of $\overline{\QQ}_{p}$ and by $\cM_{p}=\{z\in \overline{\QQ}_{p}|\ord_{p}(z)>0\}$ its maximal ideal.

The set of critical values $V_{f}$ of $f$ is a definable set in $\cL_{\text{ring}}$ given by
\[
z\in V_{f}\Leftrightarrow\exists y\Big[z=f(y)\wedge\dfrac{\partial f}{\partial x_{1}}(y)=0 \wedge \ldots \wedge \dfrac{\partial f}{\partial x_{n}}(y)=0\Big].
\]
By elimination of quantifiers in the $\text{ACF}_{0}$-theory, i.e., the theory of algebraically closed fields of characteristic $0$, and because of the fact that $V_{f}$ is a finite set, there exist non-zero polynomials $T(z)\in\ZZ[z]$ and $R(z)\in\overline{\QQ}[z]$, such that $V_{f}=Z(R)\subset Z(T)$. Moreover, we can assume that $T(z)$ and $R(z)$ only have simple roots in $\overline{\QQ}$. By logical compactness, there exists $N_0$, such that for all $p >N_0$, $T_{p}(z)\in\FF_{p}(z)$ and $R_{p}(z)\in\overline{\FF}_{p}(z)$ also only have simple roots in $\overline{\FF}_{p}$ and $V_{\overline{f}}=Z(R_{p})\subset Z(T_{p})\subset\overline{\FF}_{p}$, where $T_{p} := (T \mod p)$ and $R_{p}:= (R \mod \cM_{p})$. Since $V_{f,p} \subset V_f$, we have $\#(V_{f,p})\leq \#(V_{f})=\deg(R) =: d$. Because $ Z(T)\subset \overline{\QQ}$ is a finite set of algebraic numbers, there exists $N\geq N_{0}$, such that for all $p>N$, the conditions (1) and (2) are satisfied, not only for $V_{f,p}$, but for $Z(T)$ and $Z(R)$ as well. 

To prove condition (3), we take $p>N$ and  $x\in\ZZ_{p}^{n}$ such that $\ord_{p}(f(x)-z)=0$ for all $z\in V_{f,p}$. Then $f(x)\notin V_{f,p}$, so $x$ is a regular point of $f$. Suppose, for a contradiction, that $\overline{x}$ is a critical point of $\overline{f}$, then $z':=\overline{f}(\overline{x})\in V_{\overline{f}}=Z(R_{p})\subset Z(T_{p})$. From the facts that $T_{p}$ has only simple roots in $\overline{\FF}_{p}$, $z'\in \FF_{p}$ and $T_{p}(z')=0$, it follows by Hensel's lemma that there exists $z_{1}\in\ZZ_{p}$ such that $T(z_{1})=0$ and $\overline{z}_{1}=z'$. Hence $\ord_p(f(x)-z_1) > 0$, and therefore $z_1 \notin V_{f,p}$. On the other hand, $R_{p}$ has also only simple roots in $\overline{\FF}_{p}$ and $z'\in Z(R_{p})$, so, by Hensel's lemma, there exists $z_{2}\in \cO_{p}$ such that $R(z_{2})=0$ and $\overline{z}_{2}=z'$. From the facts that $z_{1}$ and $z_{2}$ are both roots of $T$, $\overline{z}_1 = z'= \overline{z}_2$ and the conditions (1) and (2) are true for $Z(T)$, it follows that $z_{1}=z_{2}$. Hence $z_{1}\in Z(R) = V_f$, and we knew already that $z_{1}\in\ZZ_{p}$ so $z_{1}\in V_{f,p}$. This contradiction proves that condition (3) also holds.
\end{proof}

\begin{proof}[Proof of the Main Theorem \ref{**}]
Let $N, d$ be as in Lemma \ref{20} and write $V_f = \{z_1, \ldots, z_d\}$. We fix $p>N$, then we can assume that $V_{f,p}=\{z_{1},\ldots,z_{r}\}$ with $r\leq d$. For each $1\leq i\leq r$, we put $\Phi_{i,p}:=\11_{\{x\in\ZZ_{p}^{n} \mid \ord_{p}(f(x)-z_{i})>0\}}: \QQ_p^n \to \CC$. Because $f\in \ZZ[x_{1},\ldots,x_{n}]$ and by Lemma \ref{20} we see that $\Phi_{i,p}$ is residual, for all $1\leq i\leq r$, and that $\Supp(\Phi_{i,p}) \cap \Supp(\Phi_{j,p}) = \emptyset$, if $i \neq j$. We denote $\Phi_{0,p} := \11_{\ZZ_{p}^n}- \sum_{i=1}^{r}\Phi_{i,p}$, then $\Phi_{0,p}$ will also be residual. Now we have 
\begin{align*}
E_{p}(z,f)&=\int_{\ZZ_{p}^{n}}\Psi(zf(x))|dx|\\
&=\sum_{i=0}^{r}\int_{\ZZ_{p}^{n}}\Phi_{i,p}(x)\Psi(zf(x))|dx|\\
&=\sum_{i=1}^{r}\int_{\ZZ_{p}^{n}}\Phi_{i,p}(x)\Psi(z(f(x)-z_{i}))\exp(2\pi izz_{i})|dx|+ E_{\Phi_{0,p}}(z,p,f)\\
&=\sum_{i=1}^{r}\exp(2\pi izz_{i})E_{\Phi_{i,p}}(z,p,f_{i})+E_{\Phi_{0,p}}(z,p,f),
\end{align*}
where $f_{i}(x)=f(x)-z_{i}$ for $1\leq i\leq r$.

 If $1\leq i\leq r$, we note that $C_{\overline{f}_{i}}\cap\Supp(\overline{\Phi}_{i,p})\subset \overline{f}_{i}^{-1}(0)$. So we can use Theorems \ref{05} and \ref{06} for $f_i$. According to Remark \ref{rem: adapt proof}, the Main Theorem \ref{*} is still true for the exponential sum $E_{\Phi_{i,p}}(z,p,f_{i})$, where we take $\sigma_{i}=\min\big\{c(f_{i}),\frac{1}{2}\big\}$. In the proofs from Section \ref{sec: geometry} we need to replace $c_{I,\chi}^{0}$ by $c_{I,\Phi_{i,p},\chi}$ and $\overset{\circ}{\overline{E}}_{I}\cap\overline{h}^{-1}(0)$ by $\overset{\circ}{\overline{E}}_{I}\cap\overline{h}_{i}^{-1}(Z(\overline{f}_i))$, with $h_i: Y_i \to \QQ(z_{i})^n$ an embedded resolution for $Z(f_i)$. For each $1\leq i \leq r$, there exist a constant $C_i$ and a natural number $N_i >N$, only depending on the critical value $z_i \in V_f$ and the chosen resolution $h_i$ of $f_i$, such that, if $p> N_i$ and $z_i \in V_{f,p}$, then we have
\[
\lvert E_{\Phi_{i,p}}(z,p,f_{i}) \rvert \leq C_i m^{n-1}p^{-m\sigma_i}.
\]
We remark that by definition of $\Phi_{0,p}$ and by condition (3) from Lemma \ref{20}, we have $C_{\overline{f}}\cap\Supp(\overline{\Phi}_{0,p})=\emptyset$, for all $p >N$, and thus it is well known that $E_{\Phi_{0,p}}(z,p,f)=0$, for $|z|>p$ (see \cite{09}, Remark 4.5.3).\\

We recall that $a(f)$ is the minimum, over all $b \in \CC$, of the log-canonical thresholds of the polynomials $f(x)-b$. Therefore, if we set $\sigma=\min\big\{a(f),\frac{1}{2}\big\}$, then $\sigma \leq \underset{1\leq i\leq d}{\min}\sigma_{i}$, hence there exist a constant $C>\underset{1\leq i\leq d}{\max}C_i$ and a natural number $N'>\underset{1\leq i\leq d}{\max}N_i$, such that for all $p>N'$ and $m\geq 2$, we have 
\[
|E_{m,p}|\leq Cm^{n-1}p^{-m\sigma}.\qedhere
\]
\end{proof}

\section{The uniform version of the Main Theorem \ref{*}}\label{sec: uniform version}

In this section $f \in \ZZ[x_1, \ldots, x_n]$ is a nonconstant polynomial. We will describe how to adapt the Sections \ref{sec: model theory} and \ref{sec: geometry} to obtain a constant $C$ and a natural number $N$, such that for all $y \in \ZZ^n$ and for all $m \geq 1$, $p >N$, we have
\begin{equation}\label{eq: uniform result}
\lvert E^y_{m,p}(f)\rvert := \Big\lvert\frac{1}{p^{mn}} \sum_{\overline{x} \in \overline{y}+(p\ZZ/p^m\ZZ)^n} \exp\Big(\frac{2\pi i f(x)}{p^m}\Big) \Big\rvert \leq Cm^{n-1}p^{-m\sigma_{y,p}}.
\end{equation}
Here we take $\sigma_{y,p} = \min\{a_{y,p}(f), \frac{1}{2}\}$. We recall that $a_{y,p}(f)$ is the minimum of the log canoncial thresholds at $y'$ of the polynomials $f(x)-f(y')$, where $y'$ runs over $y+(p\ZZ_p)^n$. Notice that the case $m=1$ is covered by Remark \ref{rem: m=1}. Hence we can assume that $m \geq 2$.\\

Let $V_{f} = \{z_1, \ldots, z_d\} \subset \overline{\QQ}$ be the set of critical values of $f$ in $\overline{\QQ}$, where $d$ is as in Lemma \ref{20}. For each $1 \leq j \leq d$ we put $f_j(x):=f(x)-z_j$ and we fix an embedded resolution $(Y_j, h_j)$ of $f_j^{-1}(0)$ over $\QQ(z_{i})$. Let $N'_j$ be a natural number, such that for all $p>N'_j$, $(Y_j, h_j)$ has good reduction modulo $p$. Furthermore, let $N'_0$ be a natural number, such that for all $p>N'_0$, we have $V_{f,p}=V_{f} \cap \QQ_p \subset \ZZ_p$, any two distinct points $z, z'$ in $V_{f,p} $ satisfy $\ord_p(z-z') =0$ and if $x \in \ZZ_p^n$ such that $\ord_p(f(x)-z)=0$ for all $z \in V_{f,p}$, then $x$, resp.\ $\overline{x}$, is a regular point of $f$, resp.\ $\overline{f}$ (see Lemma \ref{20}). We put $N' := \underset{0\leq i\leq d}{\max}N'_i$  and for each $p>N'$ we consider a partition of $\ZZ^n = \bigcup_{j=0}^d A_{j,p} \cup \bigcup_{j=1}^d B_{j,p}$, where
\begin{align*}
A_{j,p} &:= \{y \in \ZZ^n \mid \ord_p(f_j(y)) >0 \text{ and $f$ has a critical point in } y+(p\ZZ_p)^n\},\\
B_{j,p} &:= \{y \in \ZZ^n \mid \ord_p(f_j(y)) >0 \text{ and $f$ has no critical points in } y+(p\ZZ_p)^n\},
\end{align*}
for $1 \leq j \leq d$, and
\[
A_{0,p} := \ZZ^n \backslash \bigcup_{j=1}^d (A_{j,p} \cup B_{j,p}). 
\]
First of all, for $p>N'$, we observe that if $y \in A_{0,p}$, then $\ord_p(f(y)-z_j) \leq 0$ for all $1 \leq j \leq d$. In particular, $\ord_p(f(y)-z_j) = 0$ for all $z_j \in V_f \cap \ZZ_p = V_{f,p}$. So $\overline{y}$ is a regular point of $\overline{f}$, by Lemma \ref{20}, hence
the condition $C_{\overline{f}} \cap \text{Supp}(\overline{\Phi}_{y,p}) = \emptyset$, with $\Phi_{y,p} := \11_{y+(p\ZZ_p)^n}$, is satisfied. Thus, by Remark 4.5.3 from \cite{09}, we get that $E^y_{m,p}(f) = 0$, for all $m \geq 2$, $p>N'$ and $y \in A_{0,p}$.\\

Secondly, if $1 \leq j \leq d$, $p>N'$, and $y \in B_{j,p}$, then $f_j$ has no critical points in $y+(p\ZZ_p)^n$. So by 1.4.1 from \cite{09}, we have $E^y_{m,p}(f_j) = 0$, for $m$ large enough. Using Corollary 1.4.5 from \cite{09}, we see that $(p^{s+1}-1)Z_{\Phi_{y,p}}(p, \chi_{\text{triv}}, s, f_j)$ and $Z_{\Phi_{y,p}}(p,\chi, s, f_j)$, for $\chi \neq \chi_{\text{triv}}$, cannot have any poles. Because the resolution $(Y_j,h_j)$ of $f_j$ has good reduction modulo $p$, for $p>N'$, and $C_{\overline{f}_j} \cap \text{Supp}(\overline{\Phi}_{y,p}) \subset \overline{f}_j^{-1}(0)$, for $y \in B_{j,p}$, the Theorem \ref{06} applies. By combining it with Proposition \ref{04}, we get that for all $p>N'$ and $y \in B_{j,p}$, the sum $E^y_{m,p}(f_j)$ equals
\begin{align}\label{total expression}
Z_{\Phi_{y,p}}(p, \chi_{\text{triv}}, 0, f_j) &+ \text{Coeff}_{t^{m-1}} \Big(\frac{(t-p)Z_{\Phi_{y,p}}(p, \chi_{\text{triv}}, s, f_j)}{(p-1)(1-t)}\Big)\\
&+ \sum_{\substack{\chi \neq \chi_{\text{triv}},\\ c(\chi)=1}} g_{\chi^{-1}} \chi(u) \text{Coeff}_{t^{m-1}}(Z_{\Phi_{y,p}}(p, \chi, s, f_j)).\nonumber
\end{align}
Since $Z_{\Phi_{y,p}}(p,\chi, s, f_j)$ does not have any poles for $\chi \neq \chi_{\text{triv}}$, we can see that, for $m$ big enough, $\text{Coeff}_{t^{m-1}}(Z_{\Phi_{y,p}}(p, \chi, s, f_j))$ will not depend on $m$. Also the total expression \ref{total expression} is independent of $m$, for $m$ big enough (because it is equal to 0). Therefore the part $\text{Coeff}_{t^{m-1}} \Big(\frac{(t-p)Z_{\Phi_{y,p}}(p, \chi_{\text{triv}}, s, f_j)}{(p-1)(1-t)}\Big)$ must be independent of $m$ as well, for $m$ big enough. This can only be the case if $\frac{(t-p)Z_{\Phi_{y,p}}(p,\chi_{\text{triv}},s,f_j)}{(p-1)(1-t)}$, as a function in $t$, has at most two poles, one pole at $t=1$ of order $1$ and  one pole at $t=0$. However, the explicit formula of $Z_{\Phi_{y,p}}(p,\chi_{\text{triv}},s,f_j)$ implies that it can not have poles at $t=0$. So $\frac{(t-p)Z_{\Phi_{y,p}}(p,\chi_{\text{triv}},s,f_j)}{(p-1)(1-t)}$ has at most one pole, and this pole (if it exists) must be of order $1$ at $t=1$.

According to 4.1.1 from \cite{09}, the degree of $Z_{\Phi_p}(p, \chi, s, f_j) \leq 0$ (as a rational function in $t$), for all $p>N'$ and all charachters $\chi$ with conductor $c(\chi)=1$. This implies that $\frac{(t-p)Z_{\Phi_{y,p}}(p,\chi_{\text{triv}},s,f_j)}{(p-1)(1-t)}$ is of the form $c + \frac{d}{1-t}$, for certain $c, d \in \CC$, and that $Z_{\Phi_{y,p}}(p, \chi, s, f_j)$ is equal to a constant function, for $\chi \neq \chi_{\text{triv}}$. Now we can easily see that for all $m\geq 2$, $\text{Coeff}_{t^{m-1}}\Big(\frac{(t-p)Z_{\Phi_{y,p}}(p,\chi_{\text{triv}},s,f_j)}{(p-1)(1-t)}\Big)$ and $\text{Coeff}_{t^{m-1}}(Z_{\Phi_{y,p}}(p, \chi, s, f_j))$, for $\chi \neq \chi_{\text{triv}}$, are indepent of $m$.
We conclude that $E^y_{m,p}(f_j) = 0$, for all $m\geq 2$, $p>N'$ and $y \in B_{j,p}$.\\

The last case is the one where $y \in A_{j,p}$, for $1 \leq j \leq d$. We will show that in this case there exists a constant $C_j$ and a natural number $N_j$ (only depending on $j$, not on $y$), such that for all $p>N_j$, $m\geq 2$ and $y \in A_{j,p}$, we have
\begin{equation}\label{eq: adapt formula}
|E^y_{m,p}(f_j)| \leq C_j m^{n-1}p^{-m\sigma_{y,p}}.
\end{equation}
By taking $N:= \max\{N', N_1, \ldots, N_d\}$ and $C:=\max\{C_1, \ldots, C_d\}$ (both independent of $y$), the formula \ref{eq: uniform result} will hold for all $y \in \ZZ^n$, $p >N$ and $m \geq 1$. In what follows, we will show how to adapt the proofs of both Sections \ref{sec: model theory} and \ref{sec: geometry}, to obtain the formula \ref{eq: adapt formula}.

\subsection{Adapting Section \ref{sec: geometry}}\label{sec: adaption 1}
If we want to be able to use the method of proof that was outlined in Section \ref{sec: geometry}, then we need to show the following result, for all $j>0$, $y \in A_{j,p}$ and $p>N'$:
\begin{equation}\label{eq: log canonical threshold reducing mod p}
a_{y,p}(f) = \min_{E : \overline{y}\in\overline{h}_j(\overline{E}(\FF_{p}))}\big\{\frac{\nu}{N}\big\},
\end{equation}
where $E$ is an irreducible component of $h_j^{-1}(Z(f_j))$ with numerical data $(N, \nu)$. When we compare this to the Formula \ref{eq: lct reduced mod p}, we see that, by replacing $c_0(f)$ by $a_{y,p}(f)$, we can adapt the results of Section \ref{sec: geometry} to $f_j$ with $\Phi_{p} = \11_{y+(p\ZZ_p)^n}$. Indeed the condition $C_{\overline{f}_j} \cap \text{Supp}(\overline{\Phi}_p) \subset \overline{f}_j^{-1}(0)$ is satisfied. This proves the Formula \ref{eq: adapt formula} and by Remark \ref{rem: adapt proof} we know that the constant $C_j$ and the natural number $N_j$ only depend on $f_j$ and the chosen resolution $(Y_j, h_j)$.

All that is left, is to prove Equation \ref{eq: log canonical threshold reducing mod p} for $y \in A_{j,p}$ and $j>0$. We remark that if $y' \in y+(p\ZZ_p)^n$ is not a critical point of $f$, then $c_{y'}(f(x)-f(y')) = 1$. If $y' \in y+(p\ZZ_p)^n$ is a critical point of $f$, then we know by Lemma \ref{20} that $f(y') = z_j$, hence $f_j(y') = f(y') - f(y') = 0$. Since $(Y_j, h_j)$ has good reduction modulo $p$, for $p>N'$, we know that, after possibly enlarging $N'$ as we did for \ref{eq: lct reduced mod p}, we have
\[
c_{y'}(f(x)-f(y')) = c_{y'}(f_j) = \min_{E:\overline{y'}\in\overline{h}_j(\overline{E}(\FF_{p}))}\big\{\frac{\nu}{N}\big\} = \min_{E:\overline{y}\in\overline{h}_j(\overline{E}(\FF_{p}))}\big\{\frac{\nu}{N}\big\} \leq 1.
\]
If $y \in A_{j,p}$, then $y+(p\ZZ_p)^n$ contains at least one critical point of $f$, in which case Equation \ref{eq: log canonical threshold reducing mod p} holds. 


\subsection{Adapting Section \ref{sec: model theory}}
For $j>0$ and $y \in A_{j,p}$, we will split the exponential sum $E^y_{m,p}(f_j)$ into three subsums in exactly the same way as in Section \ref{sec: model theory}. In each of the Lemmas \ref{14}, \ref{16}, \ref{17} and in the proof of the Main Theoreom \ref{*} from Section \ref{sec: model theory} we need to make some changes.

\begin{lem}\label{lem 1 new version}
Let $f \in \ZZ[x_1, \ldots, x_n]$ be a nonconstant polynomial and let $z_j \in V_f$ be a critical value of $f$. There exists a natural number $N_0 > N'$, such that for all $m \geq 1$, for all $p>N_0$ and for all $y \in A_{j,p}$, we have
\[
\sum_{\substack{\overline{x}\in \overline{y}+(p\ZZ/p^{m}\ZZ)^{n},\\ \ord_{p}(f_j(x))\leq m-2}}\exp\left(\frac{2\pi if_j(x)}{p^{m}}\right)=0.
\]
Remark that if $A_{j,p}\neq\emptyset$, then $z_{j}\in V_{f,p}\subset\ZZ_{p}$, so the term $\exp\left(\dfrac{2\pi if_j(x)}{p^{m}}\right)$ is well-defined.
\end{lem}
To prove this lemma, we adapt the proof of Lemma \ref{14} as follows. We replace the formula $\phi$ by
\begin{align*}
&\phi_j(x_1, \ldots, x_n, z, \xi_1, \ldots, \xi_n, m) =\\
&\bigwedge_{i=1}^n (\overline{x}_i=\xi_i) \wedge (\ord(z - z_j) \leq m-2) \wedge (\ord(z -f(x_1, \ldots, x_n)) \geq m),
\end{align*}
where $x_i, z$ are in the valued field-sort, $\xi_i$ are in the residu field-sort and $m$ is in the value group-sort. This is an $\cL_{\ZZ} \cup \{z_j\}$-formula, with $z_j$ a constant symbol in the valued field-sort. We remark that the function $\cO_K \to k_K: x \mapsto \overline{x} = (x \mod \cM_K)$ is definable in $\cL_{\text{DP}}$.

Now $\phi_j$ induces a definable subassignment $X_j \subset h[n+1, n, 1]$ and constructible functions $F_j := \11_{X_j}$ and $G_j := \pi_{!}(F_j)$, where $\pi: h[n+1, n, 1] \to h[1, n,1]$ is the projection onto the last $n+2$ coordinates. For each prime $p$, for each uniformiser $\varpi_p$ of $\QQ_p$ and for each $y \in A_{i,p}$, we have the following interpretation of $G_j$ in $\QQ_p$:
\[
G_{j, \varpi_p}(z, \overline{y}, m) = \#\{\overline{x}^{(m)} \in \overline{y}^{(m)}+(p\ZZ/p^m\ZZ)^n \mid f(x) \equiv z \mod p^m\},
\]
if $\ord_p(z - z_j) \leq m-2$, and
\[
G_{j,\varpi_p}(z, \overline{y}, m) = 0,
\]
if $\ord_p(z-z_j) \geq m-1$. Here the notation $\overline{x}^{(m)}$ means the class of $(x \mod p^m)$. Note however that $G_{j,\varpi_p}$ actually only depends on $(y \mod p)$, i.e., on $\overline{y}$. We remark that if $A_{j,p} \neq \emptyset$, then $z_j \in \QQ_p$, which makes it possible to interprete $\ord(z-z_j)$ (and other formulas that contain the symbol $z_j$) in $\QQ_p$.
We apply Corollary \ref{cell} to $G_j$ to obtain a cell decomposition where the centers $c_i$ are given by $\cL_{\ZZ}\cup\{z_j\}$-formulas $\theta_j(z,\xi, \eta, \gamma, m)$. By elimination of quantifiers, $\theta_i$ is equivalent to the formula
\[
\bigvee_{k}\Big(\zeta_{ik}\big(\ac g_1(z), \ldots, \ac g_s(z), \xi, \eta\big) \wedge \nu_{ik}\big(\ord g_1(z), \ldots, g_s(z), \gamma, m\big)\Big),
\]
where $\zeta_{ik}$ is an $\cL_{\text{ring}}$-formula and $\nu_{ik}$ an $\cL_{\text{oag}}$-formula, and $g_1, \ldots, g_s \in (\ZZ[z_{j}][[t]])[z]$. The rest of the proof of Lemma \ref{14} still applies if we replace $\ord(z)$ by $\ord(z-z_j)$ everywhere. By going over the proof, we can see that the natural number $N_0$ that is obtained in the proof, only depends on $j$.

\begin{lem}
Let $f \in \ZZ[x_1, \ldots, x_n]$ be a nonconstant polynomial and let $z_j \in V_f$ be a critical value of $f$. There exists, for each integer $m >1$, a natural number $N_m > N'$ and a positive constant $D_m$, such that for all $p >N_m$ and for all $y \in A_{j,p}$, we have
\[
\Big\lvert\sum_{\substack{\overline{x}\in \overline{y}+(p\ZZ/p^{m}\ZZ)^{n},\\ \ord_{p}(f_j(x)) = m-1}} p^{-mn}\exp\left(\frac{2\pi if_j(x)}{p^{m}}\right)\Big\rvert\leq D_m p^{-m \sigma_{y,p}}.
\]
\end{lem}

To prove this lemma, we adapt the proof of Lemma \ref{16} as follows. We replace the formulas $\phi$ and $\overline{\phi}$ by
\begin{align*}
&\phi_j(x_1, \ldots, x_n, z, \xi_1, \ldots, \xi_n, m) =\\
&\bigwedge_{i=1}^n (\overline{x}_i=\xi_i) \wedge (\ord(z - z_j) =m-1) \wedge (\ord(z - f(x_1, \ldots, x_n)) \geq m),\\
&\overline{\phi}_j(x_1, \ldots, x_n, \xi, \xi_1, \ldots, \xi_n, m) =\\
&\bigwedge_{i=1}^n (\overline{x}_i=\xi_i) \wedge (\ord(f(x_1, \ldots, x_n)-z_j)=m-1) \wedge \ac(f(x_1, \ldots, x_n)-z_{j}) = \xi),
\end{align*}
where $x_i, z$ are in the valued field-sort, $\xi_i, \xi$ are in the residue field-sort and $m$ is in the value group-sort. These are also $\cL_{\ZZ} \cup \{z_j\}$-formulas. Most of the other modifications in the proof of Lemma \ref{16} are the same as we discussed above for Lemma \ref{lem 1 new version}.

The only moment that we have to be more careful, is when estimating $\#\{\overline{x}^{(m)}\in \overline{y}^{(m)} + (p\ZZ_p/p^m\ZZ_p)^n \mid \ord_p(f_j(x)) = m-1\}$. From Section \ref{sec: adaption 1} we know that if $p>N'$ and if $y \in A_{j,p}$, then there exists $y' \in y+(p\ZZ_p)^n$, such that $a_{y,p}(f) = c_{y'}(f_j)$. By Corollary 3.6 from \cite{15}, we have
\[
a_{y,p}(f) = c_{y'}(f_j) \leq \frac{(m-1)n - \dim_{\FF_p}(\tilde{A}_{p,m,y})}{m-1},
\]
where $A_{p,m,y} := \{\overline{x}^{(m)} \in \overline{y}^{(m)} + (p\ZZ_p/p^m\ZZ_p)^n \mid \ord_p(f_j(x)) = m-1\}$, viewed as a subvariety of $\FF_p^{mn}$, and where $\tilde{A}_{p,m,y}$ is the image of $A_{p,m,y}$ under the projection $\pi_m : (\ZZ_p/p^m\ZZ_p)^n \to (\ZZ_p/p^{m-1}\ZZ_p)^n$, viewed as a subvariety of $\FF_p^{mn-n}$. Then $\#A_{p,m,y} \leq \#\tilde{A}_{p,m,y} \cdot p^n$. By the Lang-Weil estimate, there exists a constant $D'_{m,y}$, not depending on $p$, such that
\[
\#\tilde{A}_{p,m,y} = D'_{m,y}p^{\dim_{\FF_p}(\tilde{A}_{p,m,y})} + O(p^{\dim_{\FF_p}(\tilde{A}_{p,m,y})-\frac{1}{2}}).
\]
By looking at the arcspace of $Z(f_j)$, we can see that, for each $m$, there are finitely many schemes $Z_1^{(m)}, \ldots, Z_{k_m}^{(m)}$, such that for all $p$ and $y$, $\tilde{A}_{p,m,y}\cong Z_i^{(m)}(\FF_p)$ for some $i \in \{1, \ldots, k_m\}$. This means that the constant $D'_{m,y}$, which we know already to be independent of $p$, only depends on the set of schemes $\{Z_1^{(m)}, \ldots, Z_{k_m}^{(m)}\}$. Hence there exists a constant $D'_{m,j}$, such that $D'_{m,j} \geq D'_{m,y}$ for all $y \in A_{j,p}$. By going over the rest of the proof of Lemma \ref{16}, we can see that the natural number $N_m$ and the constant $D_m$, that are obtained in the proof, only depend on $m$ and $j$.\\

We need to make similar adjustments in the proof of Lemma \ref{17}, to obtain the following lemma.

\begin{lem}
Let $f \in \ZZ[x_1, \ldots, x_n]$ be a nonconstant polynomial and let $z_j \in V_f$ be a critical value of $f$. There exists, for each integer $m >1$, a natural number $N_m > N'$ and a positive constant $D_m$, such that for all $p >N_m$ and for all $y \in A_{j,p}$, we have
\[
\Big\lvert\sum_{\substack{\overline{x}\in \overline{y}+(p\ZZ/p^{m}\ZZ)^{n},\\ \ord_{p}(f_j(x))\geq m}}p^{-mn}\exp\left(\frac{2\pi if_j(x)}{p^{m}}\right)\Big\rvert\leq D_m p^{-m \sigma_{y,p}}.
\]
\end{lem}

The final step after these three lemmas, is to modify the proof of the Main Theorem \ref{*} at the end of Section \ref{sec: model theory}. According to Corollary \ref{07} and its proof, there exist natural numbers $s_j, M_{j}, N''_j$, such that for all $p>N''_j$, $m > M_j$ and $y \in A_{j,p}$, we have
\begin{equation}\label{corollary 2.6 uniform in y}
E^y_{m,p}(f_j) = \sum_{i=1}^{s_j} a_{i,p,y} m^{\beta_{ij}} p^{-\lambda_{ij}m} \11_{A_{ij}}(m).
\end{equation}
We can easily see that $\beta_{ij}$, $\lambda_{ij}$ and $A_{ij}$ only depend on $f_j$ and not on $y$. 
By going through the proof of Claim \ref{claim} we obtain a constant $C_0$ and natural numbers $\tilde{M}, \tilde{N}$ (that depend on $\beta_{ij}$, $\lambda_{ij}$ and $A_{ij}$, but not on $a_{i,p,y}$), such that for all $m>\tilde{M}$, $p>\tilde{N}$, $y \in A_{j,p}$ and $1 \leq i \leq s_j$, we have
\[
|a_{i,p,y}p^{-\lambda_{ij}m}| \leq C_0p^{-\sigma_{y,p}m}.
\]
Now \ref{eq: adapt formula} follows easily.

\end{document}